\documentclass[11pt]{amsart}

\usepackage{geometry}
\usepackage{amsmath,amssymb}
\usepackage{enumerate}
\usepackage[T1]{fontenc}
\usepackage{lmodern}
\usepackage[utf8]{inputenc}
\usepackage{afterpage}
\usepackage{float}

\newcommand{\needspace}[1]{\par\penalty-100\begingroup\dimen0=\pagegoal\advance\dimen0 by -\pagetotal
  \ifdim\dimen0<#1 \newpage\fi\endgroup}
\usepackage{xcolor}
\usepackage{tikz}
\usepackage{booktabs,longtable}
\usetikzlibrary{decorations.pathreplacing,patterns,shapes.geometric}
\definecolor{inertC}{RGB}{68,119,170}
\definecolor{caC}{RGB}{238,153,68}
\definecolor{cbC}{RGB}{34,136,51}
\definecolor{preC}{RGB}{195,195,195}
\definecolor{cutC}{RGB}{204,51,17}

\usepackage[colorlinks,linkcolor=blue!45!black,citecolor=blue!45!black,urlcolor=blue!45!black]{hyperref}
\hypersetup{pdftitle={Purely Periodic Three-move Subtraction Games},pdfauthor={Hikaru Manabe},
  pdfsubject={Combinatorial game theory},
  pdfkeywords={subtraction game, Sprague-Grundy value, periodicity, P-position, bracket sequence}}

\theoremstyle{plain}
\newtheorem{thm}{Theorem}
\newtheorem{corollary}[thm]{Corollary}
\newtheorem{prop}[thm]{Proposition}
\newtheorem{lemma}[thm]{Lemma}
\newtheorem{conj}[thm]{Conjecture}

\theoremstyle{definition}
\newtheorem{defi}[thm]{Definition}
\newtheorem{remark}[thm]{Remark}
\newtheorem{example}[thm]{Example}
\usepackage{etoolbox}
\newcommand{\eodsym}{\ensuremath{\diamond}}
\newcommand{\eodmark}{\leavevmode\unskip\nobreak\hfil\penalty50\hskip1em\hbox{}%
  \nobreak\hfil\eodsym\parfillskip=0pt\finalhyphendemerits=0\par}
\AtEndEnvironment{defi}{\eodmark}
\AtEndEnvironment{remark}{\eodmark}
\AtEndEnvironment{example}{\eodmark}

\newcommand{\Z}{\mathbb{Z}}
\newcommand{\N}{\ensuremath{\mathbb{N}}}      
\renewcommand{\P}{\ensuremath{\mathrm{P}}}      
\newcommand{\W}{\mathcal{W}}                   
\newcommand{\G}{\mathcal{G}}                   
\newcommand{\Sbase}{S\setminus\{c\}}           
\newcommand{\Th}{\Theta}                       

\renewcommand{\ge}{\geqslant}
\renewcommand{\le}{\leqslant}

\title{Purely Periodic Three-move Subtraction Games}
\author{Hikaru Manabe}
\address{College of Information Science, School of Informatics, University of Tsukuba, Japan}
\email{urakihebanam@gmail.com}

\begin{document}
\begin{abstract}
We determine the full Sprague--Grundy sequence and its least period for a
class of three-move subtraction games whose sequences are periodic from
the start. Write the move set as $S=\{a,b,c\}$, with
$0<a<b<c$ and $\gcd(a,b,c)=1$.
The case $a=1$ is known and is summarized separately.
For $a\ge2$ and $c\ne a+b$, we give three explicit sufficient tests for pure
periodicity, associated with the candidate periods $a+b$, $c+a$, and $c+b$.
For fixed $a,b$, the tests depend only on $c\bmod(a+b)$ and are computed
from the losing positions of the two-move game $\{a,b\}$.
Whenever a test succeeds, we give the losing positions in closed form,
reconstruct the remaining values, and prove that the least period is
the smallest candidate whose test succeeds.
The construction shows how the added move $c$ modifies the two-move
pattern to form a repeating block.
We also give a uniform formulation of the known additive case $c=a+b$.
For $a\ge2$, we conjecture that the tests cover every purely periodic
non-additive game. In this range, when $c\ge2(a+b)$, we prove necessity
for purely periodic games with least period at most $a+b$.
\end{abstract}
\date{September 16, 2026}
\maketitle
\enlargethispage{4pt}

\section{Introduction}\label{sec:intro}

A \emph{subtraction game} \cite{berlekamp2004winning} has positions
$x\in\N$ and legal moves $x\mapsto x-s$ for $s\in S$, $s\le x$; the
player unable to move loses. We study $S=\{a,b,c\}$ with $0<a<b<c$.
Its \emph{nim-value} $\G_S(x)$ \cite{Sprague,Grundy} is the least
non-negative integer absent from the values of its options, so
$\G_S(x)\in\{0,1,2,3\}$. The zero set $\W_0(S)$ consists of the losing,
or \emph{$\P$-positions}. The nim-sequence is eventually periodic
\cite{golomb1966mathematical}: $\G_S(x+p)=\G_S(x)$ for all $x\ge q$.
Taking $p\ge1$ minimal and then $q\ge0$ minimal defines the
\emph{period} $p(S)$ and \emph{pre-period} $q(S)$. Pure periodicity
means $q=0$; we call the case $q>0$ pre-periodic.

Our aim is to determine purely periodic nim-sequences directly from
the three moves. The case $a=1$ is already determined in
\cite[Example~2.4 and Proposition~2.5]{Zhang}; its purely periodic
formulas are recalled in Table~\ref{tab:a1-known}.
We therefore focus on primitive rulesets ($\gcd(a,b,c)=1$) with
$2\le a<b<c$ and $c\ne a+b$. For these rulesets, we give an explicit sufficient test
for each of the candidate periods $a+b$, $c+a$, and $c+b$.
We call a candidate admissible when its test succeeds.
If at least one test succeeds, we determine
the entire nim-sequence and prove that its least period is the
smallest candidate whose test succeeds
(Theorems~\ref{thm:main} and~\ref{thm:lfunction}).

The construction starts with the two-move game $B=\{a,b\}$, which
agrees with $S$ at every position below $c$. Its $\P$-positions repeat
with period $a+b$ and residue set $\mathcal Z$
(Proposition~\ref{prop:base}). When the move $c$ becomes available,
it can connect two positions in this pattern. For fixed $a,b$, the
residues it connects depend only on $\rho=c\bmod(a+b)$.
We represent $\mathcal Z$ on the circle $\Z/(a+b)\Z$ and call $\rho$
the angle of the added move. Writing $c=K(a+b)+\rho$, with
$0\le\rho<a+b$, records the number $K$ of complete base blocks
below $c$ separately from this residue information.

The three tests express conditions on this angle as avoidance of
explicit unions of arcs. When the test for $P=c+a$ or $c+b$ succeeds,
we retain the base $\P$-positions below $c$, determine the
$\P$-positions in $[c,P)$, and prove that this initial block repeats
with period $P$ (Theorems~\ref{thm:verify} and~\ref{thm:cut}).
When the test for $P=a+b$ succeeds, the base pattern is unchanged.
Ferguson's pairing gives the positions of value $1$; the $L$-function
of Theorem~\ref{thm:lfunction} distinguishes values $2$ and $3$ on
the remaining positions. Table~\ref{tab:closedforms} collects the
$\P$-position formulas and least periods.

For the \emph{additive} case $c=a+b$, with $a\ge2$ and $\gcd(a,b)=1$,
we give a uniform formulation of the known results. The $n$-th
$\P$-position, indexed from $0$, is
\[
  w_3(n)=n+a\lfloor n/a\rfloor+b\lfloor n/N\rfloor,
  \qquad N=\#\bigl(\W_0(B)\cap[0,b)\bigr).
\]
Shifting the index by $\operatorname{lcm}(a,N)$ increases both floor
functions by fixed integers, so it shifts every $\P$-position by
the same amount. Theorem~\ref{thm:additive} proves that this amount
is the least period of both the $\P$-set and the nim-sequence.
The history of these known additive results and their relation to
\cite{Larsson,MoriwakiThesis} are recorded in Remark~\ref{rem:pq}.
Only $(K,\rho)=(1,0)$ is additive: larger multiples of $a+b$ remain
subject to the angle tests.

We conjecture that the non-additive criterion is also necessary
under the stated hypotheses (Conjecture~\ref{conj:necessity}).
For $c\ge2(a+b)$, Proposition~\ref{prop:necessity-inert} proves
necessity for purely periodic games whose least period is at most
$a+b$. Other proved cases and the remaining questions are discussed
in \S\ref{sec:necessity}.

\begin{example}\label{ex:38}
For $(a,b)=(3,8)$, the base pattern on $\Z/11\Z$ is
$\mathcal Z=\{0,1,2\}\cup\{6,7\}$, and the admissible sets are
\[
  \Th_{a+b}=\{3,8\},\qquad \Th_{c+a}=\{2,8\},\qquad
  \Th_{c+b}=\{3,4,5,9,10\}.
\]
Figure~\ref{fig:circle} selects the least admissible period at each
angle. At $\rho\in\{0,1,6,7\}$ none is admissible, and the conjecture
predicts positive pre-period for non-additive $c$. The exception
$c=11$ is covered by Theorem~\ref{thm:additive}.
Figure~\ref{fig:mechanism} shows how the base pattern is cut and
completed to produce these periods.
\end{example}

\newpage
\begin{figure}[H]
\centering
\begin{tikzpicture}[scale=0.97]
  \def\R{2.5}
  \draw[gray!45] (0,0) circle (\R);
  \foreach \k in {0,1,2,6,7}{
    \draw[line width=5pt,black!14]
      ({90-\k*32.7273+16.3636}:\R) arc ({90-\k*32.7273+16.3636}:{90-\k*32.7273-16.3636}:\R);
  }
  \tikzset{mInert/.style={circle,draw=black,fill=inertC,inner sep=0pt,minimum size=8pt},
           mCa/.style={rectangle,draw=black,fill=caC,inner sep=0pt,minimum size=7.5pt},
           mCb/.style={regular polygon,regular polygon sides=3,draw=black,fill=cbC,inner sep=0pt,minimum size=9.5pt},
           mPre/.style={circle,draw=black,fill=white,inner sep=0pt,minimum size=8pt}}
  \foreach \k/\m in {0/mPre,1/mPre,2/mCa,3/mInert,4/mCb,
                     5/mCb,6/mPre,7/mPre,8/mInert,
                     9/mCb,10/mCb}{
    \node[\m] at ({90-\k*32.7273}:\R) {};
    \node at ({90-\k*32.7273}:{\R+0.45}) {\small$\k$};
  }
  \node[align=left,anchor=west] at (3.6,1.2)
    {\footnotesize
     \tikz\node[mInert,minimum size=6pt] {};\ $p=a+b$\\[2pt]
     \tikz\node[mCa,minimum size=6pt] {};\ $p=c+a$\\[2pt]
     \tikz\node[mCb,minimum size=7.5pt] {};\ $p=c+b$\\[2pt]
     \tikz\node[mPre,minimum size=6pt] {};\ $q>0$ \textup{(}conj.\textup{)}\\[6pt]
     shaded arcs: $\mathcal Z$};
\end{tikzpicture}
\caption{The base arcs and proved least periods for $(a,b)=(3,8)$.
Open dots denote the conjectured pre-periodic angles for non-additive
$c$; the additive ruleset $c=11$ at angle $0$ is purely periodic.}
\label{fig:circle}
\end{figure}
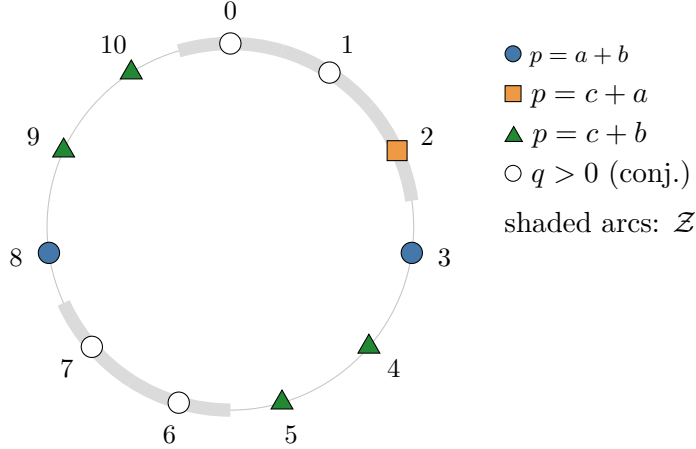

\begin{remark}
Flammenkamp \cite{Flammenkamp}, following Alth\"ofer and B\"ultermann
\cite{AB95}, found the pair-sum pattern computationally at the
outcome level. His divisibility and gcd conjecture
\cite[Vermutung~6]{Flammenkamp} predates Ward's corresponding
nim-value conjecture \cite{Ward}; outcome and nim-value periods must
be distinguished. Ho \cite{Ho} proves period and pre-period formulas
for individual three-move families, some lying in our covered regions;
Zhang \cite[Remark~4]{Zhang} corrects some of the small-$c$ data.

Zhang \cite{Zhang} already studies a fixed base with a varying extra
move and conjectures piecewise-linear periods and pre-periods
\cite[Conjecture~1.3]{Zhang}. His Propositions~3.1 and~4.1 completely
determine the slices $b=2a$ and $b=a+1$. Our proofs give an alternative derivation of his purely periodic cases.
For the families with positive pre-period discussed in \S\ref{sec:necessity},
we use Zhang's results, including \cite[Theorem~6.3]{Zhang} on ultimately
bipartite games.

Moriwaki \cite[Theorem~3]{Moriwaki} treats positive real parameters.
After normalizing $(a,b,c)$ to $(a/c,b/c,1)$, his periodic regions
contain the non-additive integral cases with $b\le2a$ covered here. On these
regions, our formulas recover his periodicity and nim-value results
through the construction of the $\P$-positions. His fourth region excludes the three candidate
periods; it does not assert the absence of every possible pure
period. Our contribution relative to these results is the rotation
test for arbitrary integral base pairs, including $b>2a$, the explicit block construction, uniform
nim-value reconstruction, and the partial necessity result above.

Mikl\'os and Post's translating-zeros criterion
\cite[Lemmas~2.12 and~2.13]{MiklosPost} is a finite test for a prescribed
outcome period. Our criterion evaluates such a test arithmetically
from $(a,b,c)$ and then reconstructs the nim-values. We use their
two-move description in Proposition~\ref{prop:base}; their $a=1$ and
additive results are discussed in Remarks~\ref{rem:a1} and~\ref{rem:pq}.
The proposed relation to their Conjecture~3 is not proved
(Remark~\ref{rem:MP-conj3}). For a survey, see
Larsson and Saha \cite{LarssonSaha}.
\end{remark}

Section~\ref{sec:prelim} gives the base pattern and difference tools.
Section~\ref{sec:rotation} constructs the non-additive $\P$-sets and
evaluates the angle tests; \S\ref{sec:additive} proves the uniform
additive formula. Section~\ref{sec:lfunction} reconstructs the
nim-values and proves minimality, and \S\ref{sec:necessity} discusses
necessity. Appendices~\ref{app:intervals} and~\ref{app:additive} contain
the interval calculations and the additive value proof.

\section{Preliminaries}\label{sec:prelim}

We use $\N=\{0,1,2,\ldots\}$ and integer intervals
$[u,v]=\{u,\ldots,v\}$, $[u,v)=\{u,\ldots,v-1\}$.
Let $B=\Sbase=\{a,b\}$ and $\rho=c\bmod(a+b)$.
For $X\subseteq\N$ and $y\in\N$, scalar translates are
$X+y=y+X=\{x+y:x\in X\}$ and
$X-y=\{x-y:x\in X,\ x\ge y\}$; subtraction thus retains only
non-negative positions.

\subsection*{Recurrence and dilation}

\begin{defi}\label{def:recurrence}
A set $T\subseteq\N$ \emph{satisfies the $S$-recurrence}, for finite
$S\subseteq\N\setminus\{0\}$, if
\begin{equation}\label{eq:precur}
  x\in T\iff\bigl(x-s\notin T\text{ for every }s\in S
  \text{ with }s\le x\bigr)
  \qquad(x\in\N).
\end{equation}
With the convention $y\notin T$ for $y<0$, this is equivalently
$x\in T\iff\bigwedge_{s\in S}(x-s\notin T)$.
\end{defi}

\begin{lemma}\label{lem:unique}
For every finite $S\subseteq\N\setminus\{0\}$ the unique set satisfying
the $S$-recurrence is $\W_0(S)=\{x:\G_S(x)=0\}$.
\end{lemma}
\begin{proof}
The mex rule gives \eqref{eq:precur} for $\W_0(S)$. Two solutions
cannot first differ at $x$, since all options are smaller than $x$
and therefore give the same right-hand side there.
\end{proof}

We call \eqref{eq:precur} the two-move or three-move recurrence when
$S=\{a,b\}$ or $\{a,b,c\}$, respectively.

\begin{lemma}\label{lem:gcd}
For $S'=\{a',b',c'\}$ and an integer $g\ge1$, the dilated ruleset $gS'$ satisfies
\[
  \G_{gS'}(x)=\G_{S'}(\lfloor x/g\rfloor),\qquad
  q(gS')=gq(S'),\qquad p(gS')=gp(S'),
\]
and $\W_0(gS')=\{gm+r:m\in\W_0(S'),\ 0\le r<g\}$.
Hence division by $\gcd(a,b,c)$ reduces pure periodicity to primitive
rulesets.
\end{lemma}
\begin{proof}
Each residue class modulo $g$ is a copy of the game $S'$ under
$x\mapsto\lfloor x/g\rfloor$, proving the nim-value and zero-set
formulas. The outcome version is \cite[Proposition~2.7]{MiklosPost}.
Write $A=\G_{S'}$, $D=\G_{gS'}$, $p'=p(S')$, $q'=q(S')$.
The tail of $A$ is not constant, since a constant value would also
occur among its options and could not be their mex. Thus $D$ has
arbitrarily late changes, all at multiples of $g$. Any eventual
period $t$ carries such a change to another, so $g\mid t$ and $t/g$
is an eventual period of $A$. Since $gp'$ is a period of $D$ from
$gq'$, its least period is $gp'$. For this period, agreement at $x$
is equivalent to $A(\lfloor x/g\rfloor+p')=A(\lfloor x/g\rfloor)$.
If $q'>0$, minimality gives failure at $x=gq'-1$; if $q'=0$, it holds from $0$.
Hence the least pre-period is $gq'$.
\end{proof}

\begin{remark}\label{rem:a1}
The case $a=1$ is already known at the level of the full
Sprague--Grundy sequence. Zhang
\cite[Example~2.4 and Proposition~2.5]{Zhang} gives the formulas for
all $1<b<c$, including the odd-$b$, even-$c$ formula of Ho
\cite[Theorem~4]{Ho}. Table~\ref{tab:a1-known} collects the purely
periodic cases and their $\P$-position formulas.
Mikl\'os and Post give the corresponding outcome classification
\cite[Theorem~4.1]{MiklosPost}; zero outcome pre-period and outcome
period dividing $b+c$ are equivalent to $b+c-1$ being winning
\cite[Corollary~2.12.2]{MiklosPost}.
Our main results concern $a\ge2$; their least-period assertion does
not extend unchanged to $a=1$ (Remark~\ref{rem:a1-fails}).
Henceforth $a\ge2$ unless stated otherwise.
\end{remark}

\subsection*{From one move to two}

To describe the blocks of the two-move pattern, write
\[
  \delta=b-a=\eta a+\varepsilon,\quad
  \eta=\lfloor\delta/a\rfloor,\quad 0\le\varepsilon<a.
\]
The one-move game $\{a\}$ has alternating losing and winning blocks
of length $a$. Its increasing $\P$-position list and the number of
terms below an integer $t\ge0$ are
\[
  w_1(n)=n+a\lfloor n/a\rfloor,\qquad
  C_a(t)=a\Big\lfloor\frac{t}{2a}\Big\rfloor+\min(a,t\bmod2a).
\]
The second move cuts this pattern at $b$ and repeats the retained
$N=C_a(b)$ positions with step $a+b$. Here $x$ denotes a position
and $n$ its index, starting at $0$.

\begin{prop}\label{prop:base}
Put $w_n=w_1(n)$ and
\[
  N=C_a(b)=\begin{cases}
    (\eta/2+1)a,&\eta\text{ even},\\
    \delta-(\eta-1)a/2,&\eta\text{ odd}.
  \end{cases}
\]
The base game $B=\{a,b\}$ has $q=0$, and $a+b$ is a period of $\G_B$.
Moreover,
\begin{align*}
  \mathcal Z&:=\W_0(B)\cap[0,a+b)=\{w_n:0\le n<N\},\\
  w_2(n)&:=w_B(n)=(a+b)\lfloor n/N\rfloor+w_1(n\bmod N).
\end{align*}
For even $\eta$, this simplifies to
$w_B(n)=n+a\lfloor n/a\rfloor+\varepsilon\lfloor n/N\rfloor$.
The set $\mathcal Z$ consists of the disjoint blocks
$[2ja,(2j+1)a)$ for $0\le j\le\lfloor\eta/2\rfloor$, together with
$[(\eta+1)a,(\eta+1)a+\varepsilon)$ when $\eta$ is odd.
In particular, $\mathcal Z=[0,a)$ for $b<2a$, and always
$\max\mathcal Z\le b-1$.
\end{prop}
\begin{proof}
The outcome periodicity and period word are
\cite[Theorems~3.3 and~3.4]{MiklosPost}, the latter attributed there to
\cite[vol.~3, p.~530]{berlekamp2004winning}. With quotient $\eta+1$
and remainder $\varepsilon$ on dividing $b$ by $a$, that word is
\[
  (0^a1^a)^{(\eta+1)/2}0^\varepsilon1^a\quad(\eta\text{ odd}),
  \qquad (0^a1^a)^{(\eta+2)/2}1^\varepsilon\quad(\eta\text{ even}),
\]
where $0$ marks a $\P$-position. These zeros give the stated blocks.
Since $w_1(ja+r)=2ja+r$ for $0\le r<a$, their increasing list is
the first $N$ terms of $w_1$; repeating it gives $w_B$.
For even $\eta$, $a\mid N$ and $a+b-2N=\varepsilon$, giving the
simplification. Finally, a two-move game has values only $0,1,2$.
Ferguson's pairing (Lemma~\ref{lem:ferguson}, independent of this
proposition) identifies the value-$1$ set as $\W_0(B)+a$.
Because $\max\mathcal Z+a<a+b$, both this set and $\W_0(B)$ repeat
from $0$ with period $a+b$, as does their complement, the value-$2$ set.
\end{proof}

\begin{remark}\label{rem:base-period}
The least outcome period is $2a$ when $b$ is an odd multiple of $a$,
and $a+b$ otherwise \cite[Theorem~3.4]{MiklosPost}.
The constructions use only the period $a+b$; this least-period
distinction enters the inert minimality argument.
\end{remark}

\begin{remark}\label{rem:Zfacts}
The block description gives three facts used below.
\textup{(i)} $\mathcal Z\cap(\mathcal Z+a)=\varnothing$, and
$\mathcal Z+a\subseteq[0,a+b)$.
\textup{(ii)} There is a disjoint partition
$[0,a+b)=\mathcal Z\sqcup(\mathcal Z+a)\sqcup U$, where
\[
  U=\begin{cases}
    [(\eta+2)a,a+b),&\eta\text{ even},\\
    [(\eta+1)a+\varepsilon,(\eta+2)a),&\eta\text{ odd}.
  \end{cases}
\]
Moreover, $(U+a)\bmod(a+b)\subseteq[0,a)\subseteq\mathcal Z$.
\textup{(iii)} Since $b=(\eta+1)a+\varepsilon\equiv-a\pmod{a+b}$, subtracting $b$
on the circle is adding $a$. In particular,
\begin{equation}\label{eq:Wcoll}
  \W_0(B)\cap(\W_0(B)+a)=\varnothing
  =\W_0(B)\cap(\W_0(B)+b).
\end{equation}
These are the collision properties of the base pattern.
\end{remark}

\begin{lemma}[Ferguson's pairing property \cite{Ferguson}]\label{lem:ferguson}
$\G_S(x)=0\iff\G_S(x+a)=1$. Hence the value-$1$ set is
$\W_1=\W_0+a$.
\end{lemma}
\begin{proof}
This is Ferguson's pairing property for the least move $a=\min S$
\cite{Ferguson,berlekamp2004winning}. It applies to every subtraction
game, including the base $B$.
\end{proof}

\begin{lemma}\label{lem:descent}
If $\G_B(y)=2$, then $y\ge2a$ and $y-2a\in\W_0(B)$.
\end{lemma}
\begin{proof}
Both options must exist and have values $0,1$. The value at $y-a$
cannot be $0$, since pairing would give $\G_B(y)=1$.
Thus $\G_B(y-a)=1$, so $y-a\in\W_0(B)+a$.
\end{proof}

\subsection*{Differences and avoidance}

An \emph{arc} on $\Z/(a+b)\Z$ is the image of an integer interval of
length less than $a+b$. We use integer endpoints and reduce modulo
$a+b$. For sets on the circle, $Y-X=\{y-x:y\in Y,x\in X\}$;
unlike scalar subtraction in $\N$, no truncation occurs.

\begin{lemma}[Avoidance]\label{lem:avoid}
\textup{(i)} For $X,Y\subseteq\Z/(a+b)\Z$ and $\rho\in\Z/(a+b)\Z$,
\[
  (\rho+X)\cap Y=\varnothing\iff\rho\notin Y-X.
\]
\textup{(ii)} If $X=\bigsqcup_i[x_i,x_i']$ and $Y=\bigsqcup_j[y_j,y_j']$ are
unions of arcs with endpoints included, then
\[
  Y-X=\bigcup_{i,j}[y_j-x_i',y_j'-x_i].
\]
Each interval on the right is reduced modulo $a+b$ and has length
$|[x_i,x_i']|+|[y_j,y_j']|-1$. It is an arc when the sum of the two
lengths is at most $a+b$, and covers the circle otherwise.
When all sums are at most $a+b$, as in our applications, avoidance
is therefore the complement of at most one arc per pair of input arcs.
\end{lemma}
\begin{proof}
A collision means $\rho+x=y$, equivalently $\rho=y-x$.
The integer differences of two intervals fill exactly the interval
between their extreme differences; take their images on the circle
and then the union over all pairs.
\end{proof}

\begin{lemma}[Index form]\label{lem:diffspec}
For $x\in\Z$ and integers $d\ge0$, $m\ge1$,
\[
  \Big\lfloor\frac{x}{m}\Big\rfloor-
  \Big\lfloor\frac{x-d}{m}\Big\rfloor
  =\Big\lfloor\frac{d}{m}\Big\rfloor+\chi,\qquad
  \chi=[x\bmod m<d\bmod m]\in\{0,1\}.
\]
Here $[A]$ is $1$ if the statement $A$ holds and $0$ otherwise.
Thus for $w(n)=n+\sum_i k_i\lfloor(n+s_i)/d_i\rfloor$, with integers
$k_i,d_i\ge1$ and $0\le s_i<d_i$, the difference $w(n)-w(n-d)$
for $n\ge d\ge0$ is $d+\sum_i k_i\lfloor d/d_i\rfloor$ plus a
subset sum of the $k_i$. The subset is specified by arcs modulo $d_i$.
\end{lemma}
\begin{proof}
Write $x=um+v$ and $d=qm+r$, with $0\le v,r<m$.
Then $\lfloor(x-d)/m\rfloor=u-q-[v<r]$.
Apply this identity term by term, with
$\chi_i=[(n+s_i)\bmod d_i<d\bmod d_i]$.
\end{proof}

\begin{lemma}[Inert differences]\label{lem:idiff}
The set $\mathcal I=\{m\in[0,a+b):
\mathcal Z\cap((\mathcal Z+m)\bmod(a+b))=\varnothing\}$ is
\[
  \mathcal I=\begin{cases}
    [a,b],&b\le2a,\\
    \{ja:1\le j\le k,\ j\text{ odd}\},&b=ka,\ k\text{ odd},\\
    \{a,b\},&\text{otherwise}.
  \end{cases}
\]
\end{lemma}
By Lemma~\ref{lem:avoid}, $\mathcal I$ is the complement of
$\mathcal Z-\mathcal Z$. Its computation from the blocks of
Proposition~\ref{prop:base} is in Appendix~\ref{app:intervals},
together with the companion gap-window calculation.

\section{The rotation criterion}\label{sec:rotation}\label{sec:sieve}\label{sec:nonadditive}\label{sec:char}

Adding $c$ to $B=\{a,b\}$ leaves the game unchanged below $c$.
We determine the first affected interval and ask when the resulting block
can repeat. For $P=c+a$ or $c+b$, two moves are opposite modulo $P$,
so only one complementary move difference remains to be excluded at a join.
The resulting conditions depend on the residue of $c$ in the base period.

Write $W=\W_0(B)$ and let $P$ denote a candidate period. Recall that
$W$ has period $a+b$ and residue set $\mathcal Z$ of size $N$. Put
\[
 \rho=c\bmod(a+b)\quad\text{(the \emph{angle})},\qquad
 K=\lfloor c/(a+b)\rfloor\quad\text{(the \emph{winding number})}.
\]
Thus $c=K(a+b)+\rho$. Translates of $\mathcal Z$ are read modulo $a+b$.
After establishing a period, we enumerate the $\P$-positions and use
\S\ref{sec:lfunction} to recover the other nim-values and prove minimality.

\subsection*{The first interval affected by \texorpdfstring{$c$}{c}}

Let $Y(x)$ be the indicator of $\W_0(S)$, and let $z(x)$ be the indicator
of $W$. The two games agree below $c$, so $Y(x)=z(x)$ for $0\le x<c$.
In fact the entire interval $[c,c+b)$ is determined by the base game:
\[
  Y(c+u)=(1-z(u))(1-z(c+u-b))\qquad(0\le u<b).
\]
To see this, first suppose $z(u)=1$. The $c$-option is a $\P$-position,
so $Y(c+u)=0$. If instead $z(u)=0$, then $u\ge a$ and $z(u-a)=1$:
below $b$ the only available move in the base game is $a$. The position
$c+u-a$ is therefore an $\mathcal N$-position, since its $c$-move reaches
$u-a$. The $c$-option $u$ is also an $\mathcal N$-position. The remaining
option, $c+u-b$, lies in $[0,c)$, where the two games agree. Hence $c+u$
is a $\P$-position exactly when this $b$-option is an $\mathcal N$-position,
which proves the formula. In particular,
\[
  \W_0(S)\cap[0,c)=W\cap[0,c),\qquad
  \W_0(S)\cap[c,c+a)=\varnothing.
\]
No periodicity of $S$ has been assumed.

Every base $\P$-position in $[c,c+b)$ survives unless its $c$-option is
also a $\P$-position. Indeed, $z(c+u)=1$ forces $z(c+u-b)=0$ by the
base recurrence. But a base $\mathcal N$-position can become a $\P$-position
if both its $b$- and $c$-options are $\mathcal N$-positions. This possibility
must be excluded before the longer initial block can be described simply
by deleting positions from the base pattern.

\subsection*{The sieve and the repetition test}

The following definition repeats that deletion pattern with a proposed
step $P$. Its block is already the actual initial block when $P=c+a$;
for $P=c+b$, the sufficient criterion below will first establish this
agreement and then check the join.

\begin{defi}\label{def:sieve}
For $P\ge c$ put
\[
  T_P=T_P(a,b,c):=\big\{\,x\in\N:\ (x\bmod P)\in\W_0(B)\ \text{and}\ (x\bmod P)-c\notin\W_0(B)\,\big\},
\]
with the convention $y\notin\W_0(B)$ whenever $y<0$.
\end{defi}

\begin{lemma}\label{lem:sieve}
Let $P\ge c$. If $T_P$ satisfies the three-move recurrence
\[
  x\in T_P\iff(x-a\notin T_P)\wedge(x-b\notin T_P)\wedge(x-c\notin T_P)
\]
for every $x$ in the interval $[c,\,c+P)$, then $\W_0(S)=T_P$; in particular the
$\P$-positions of $S$ are purely periodic with period $P$.
\end{lemma}

\begin{proof}
\emph{Range $x<c$.} Here $\G_S=\G_B$, since $c$ is unavailable, and
Definition~\ref{def:sieve} gives $T_P\cap[0,c)=W\cap[0,c)$.
\emph{Range $x\ge c$.} Write $x=x'+kP$ with $c\le x'<c+P$ and $k\ge0$;
all three options at $x'$ are nonnegative, so $P$-periodicity transfers
the tested recurrence to $x$.
The recurrence therefore holds everywhere, and Lemma~\ref{lem:unique} applies.
\end{proof}

\begin{remark}\label{rem:MP-criterion}
Mikl\'os and Post's translating-zeros test
\cite[Lemmas 2.12--2.13 and Corollary 2.12.1]{MiklosPost} applies to the
computed outcome sequence. Here the proposed block is described directly
from $(a,b,c)$; below we identify it with the actual initial block and
check the join.
\end{remark}

Let $D=[a,b)\setminus W=[a,b)\setminus\mathcal Z$ be the base gaps
below $b$. Their interval decomposition follows from Proposition~\ref{prop:base}.

\begin{prop}\label{prop:sieveform}
Let $P\in\{c+a,\,c+b\}$, and let $T_P$ be the $c$-sieve of
Definition~\ref{def:sieve}. Then
\[
  T_P\cap[0,P)=\W_0(B)\cap J_P,\qquad
  J_P=\begin{cases}
   [0,c), & P=c+a,\\[2pt]
   [0,c)\,\cup\,(c+D), & P=c+b.
  \end{cases}
\]
If $\W_0(S)=T_P$, this is the initial $\P$-block of $S$.

\end{prop}

\begin{proof}
Every position below the least move is a $\P$-position:
\begin{equation}\label{eq:lowblock}
 [0,a)\subseteq W.
\end{equation}
In Definition~\ref{def:sieve} the deletion condition is vacuous below $c$.
Above $c$, it excludes all of $[c,c+a)$ by \eqref{eq:lowblock}; for $P=c+b$, on
$[c+a,c+b)$ it retains exactly those $x\in W$ with $x-c\in D$.
\end{proof}

\emph{Repeating a known initial block.}
For the two constructions below, it is useful to separate the initial
calculation from the check at the join. Suppose $P>c$ and
$H=\W_0(S)\cap[0,P)$ has already been determined. If no two elements
of $H$ differ by any of
\[
  P-a,\qquad P-b,\qquad P-c,
\]
then $\W_0(S)=H+P\N$.

Indeed, put $T=H+P\N$. No two elements of $H$ differ by a move, since
$H$ is part of the actual $\P$-set. As every move is smaller than $P$,
a move between two elements of $T$ would give either such a difference
inside $H$ or one of the complementary differences just excluded.
Thus no move from $T$ reaches $T$. Conversely, if $x=r+kP\notin T$,
with $0\le r<P$ and $k\ge0$, the actual $\mathcal N$-position $r$ has an
available move $s\le r$ to a position $r-s\in H$. The same move sends
$x$ to $r-s+kP\in T$. These are the two parts of the three-move
recurrence, so Lemma~\ref{lem:unique} gives $\W_0(S)=T$.

For $P=c+a$, the complementary differences $P-a=c$ and $P-c=a$
are already move differences, leaving only $P-b=c-\delta$ to check.
For $P=c+b$, the complementary differences $P-b=c$ and $P-c=b$
are already excluded, leaving only $P-a=c+\delta$. In the latter case
we must first establish that the deletion pattern is the actual initial
block, as the first-block calculation above explains.

\subsection*{Admissible angles}

We now express these conditions in terms of the base pattern.
Besides the gaps $D$, the $c+b$ case uses the arc $\Lambda$ below.
Lemma~\ref{lem:T2} shows that $\Lambda$ consists of the residues $w$
for which both $w$ and $w+a$ lie outside $\mathcal Z$, with residues
read modulo $a+b$.

\begin{equation}\label{eq:I1Lambda}
  \Lambda=\begin{cases}
    \big[(\eta+1)a,\ (\eta+1)a+\varepsilon\big), & \eta\ \text{even},\\[2pt]
    \big[\eta a+\varepsilon,\ (\eta+1)a\big),    & \eta\ \text{odd}.
  \end{cases}
\end{equation}

\begin{defi}\label{def:theta}
Let $\rho\in\Z/(a+b)\Z$. We say that $\rho$ is \emph{admissible for the period $P$}, and write
$\rho\in\Th_P$, in the following three cases.
\begin{align}
  \rho\in\Th_{a+b}\ &:\iff\ \mathcal Z\cap(\mathcal Z+\rho)=\varnothing; \tag{T0}\label{eq:T0}\\
  \rho\in\Th_{c+a}\ &:\iff\ \mathcal Z\cap[0,\delta)\cap\big(\mathcal Z+(\delta-\rho)\big)=\varnothing;
    \tag{T1}\label{eq:T1}\\
  \rho\in\Th_{c+b}\ &:\iff\ (\rho+\Lambda)\cap\mathcal Z=\varnothing\ \text{ and }\
    (\rho+D)\cap\Lambda=\varnothing. \tag{T2}\label{eq:T2}
\end{align}
All three are conditions on the angle alone; none of them involves $c$ beyond its residue.
\end{defi}

The avoidance lemma converts these conditions into difference sets.

\begin{prop}\label{prop:instances}
Put $A:=\mathcal Z\cap[0,\delta)$. Then, for every $\rho\in\Z/(a+b)\Z$,
\[
  \rho\in\Th_{a+b}\iff\rho\notin\mathcal Z-\mathcal Z,\qquad
  \rho\in\Th_{c+a}\iff\rho\notin\delta-\big(A-\mathcal Z\big),
\]
\[
  \rho\in\Th_{c+b}\iff\rho\notin\big(\mathcal Z-\Lambda\big)\cup\big(\Lambda- D\big).
\]
\end{prop}

\begin{proof}
Apply Lemma~\ref{lem:avoid} to the translated pairs
$(\mathcal Z,\mathcal Z,\rho)$, $(\mathcal Z,A,\delta-\rho)$,
$(\Lambda,\mathcal Z,\rho)$ and $(D,\Lambda,\rho)$, respectively.
\end{proof}

\begin{remark}\label{rem:programme}
The four difference sets are evaluated in Lemma~\ref{lem:idiff},
Proposition~\ref{prop:arcdiff} and Theorem~\ref{thm:arcs}, with computations
in Appendix~\ref{app:intervals}. In particular,
\begin{equation}\label{eq:ThI}
 \Th_{a+b}=\mathcal I.
\end{equation}
The even harmonic $c+b$ case is treated by Proposition~\ref{prop:evenharm}.
\end{remark}

For $c+b$, condition \textup{(K2)} below first identifies the initial
block with the sieve; \textup{(K1)} then excludes a collision at the join.

\begin{lemma}\label{lem:T2}
Let $c>b$ with $\rho=c\bmod(a+b)$. Then $\rho\in\Th_{c+b}$ if and only if both
\begin{align}
  &\textup{(K1)}\quad
  \forall r\in[0,a):\ (r+\rho+\delta)\bmod(a+b)\in\mathcal Z
  \ \Longrightarrow\ (r+\delta)\bmod(a+b)\in\mathcal Z,\label{eq:K1}\\
  &\textup{(K2)}\quad
  \forall x\in[c+a,c+b):\ \big(x\notin\W_0(B)\ \wedge\ x-c\notin\W_0(B)\big)
  \ \Longrightarrow\ x-b\in\W_0(B).\label{eq:K2}
\end{align}
\end{lemma}

\begin{proof}
We first record two descriptions of the same obstruction on the circle:
\[
 \Lambda=[\delta,b)\setminus\mathcal Z
 =\{w:\ w\notin\mathcal Z,\ w+a\notin\mathcal Z\}.
\]
For the first identity, compare $[\delta,b)$ with the last block of
$\mathcal Z$ in Proposition~\ref{prop:base}; the uncovered interval has
exactly the endpoints in \eqref{eq:I1Lambda}.
For the second, recall
$[0,a+b)=\mathcal Z\sqcup(\mathcal Z+a)\sqcup U$ and
$U+a\subseteq\mathcal Z$ (Remark~\ref{rem:Zfacts}). Adding $a$ sends
each full odd block except the last into the next even block of
$\mathcal Z$; for odd $\eta$, the terminal partial block wraps into
$[0,a)$. In the last full odd block, the portion that does not return
to $\mathcal Z$ is $[(\eta+1)a,(\eta+1)a+\varepsilon)$ for even $\eta$,
and $[\eta a+\varepsilon,(\eta+1)a)$ for odd $\eta$; these are again
the two cases of \eqref{eq:I1Lambda}.

As $r$ ranges over $[0,a)$, $r+\delta$ ranges over $[\delta,b)$.
Thus \textup{(K1)} says that none of the points in
$[\delta,b)\setminus\mathcal Z=\Lambda$ enters $\mathcal Z$ after
translation by $\rho$, or $(\rho+\Lambda)\cap\mathcal Z=\varnothing$.
For \textup{(K2)}, put $u=x-c\in[a,b)$ and $w=\rho+u$ on the circle.
Its failure means $u\in D$, $w\notin\mathcal Z$ and
$w+a\notin\mathcal Z$, since $-b\equiv a\pmod{a+b}$.
The second identity makes this equivalent to
$(\rho+D)\cap\Lambda\ne\varnothing$.
These are precisely the two conditions in \eqref{eq:T2}.
\end{proof}

\subsection*{The sufficient criterion}

For the remainder of this section, assume $c\ne a+b$ and
$\gcd(a,b,c)=1$ unless stated otherwise; we retain $a\ge2$.
The next theorem establishes the $\P$-set. Its proof does not use
primitivity or $a\ge2$; these hypotheses are retained for the later
least-period classification. The remaining nim-values and least periods
will follow in \S\ref{sec:lfunction}.

\needspace{10\baselineskip}
\begin{thm}[Sufficient criterion]\label{thm:verify}
Let $c>b$, $c\neq a+b$, $\gcd(a,b,c)=1$ and $a\ge2$, and let
$P\in\{a+b,\,c+a,\,c+b\}$ be a candidate period whose condition holds, that is,
$\rho\in\Th_P$. Then
\[
  \W_0(S)=\begin{cases}\W_0(B), & P=a+b,\\[2pt] T_P, & P\in\{c+a,\,c+b\},\end{cases}
\]
with $T_P$ the sieve of Definition~\ref{def:sieve}. In particular the set of
$\P$-positions is $P$-periodic from $0$: the outcome sequence is purely periodic
with $P$ as a period.
\end{thm}

\begin{proof}
Write $W=\W_0(B)$. In the two cases $P>c$, we first identify the actual initial block.
The repetition argument then verifies the recurrence required by
Lemma~\ref{lem:sieve}.

\emph{The unchanged base pattern: $P=a+b$.}
Condition \eqref{eq:T0} excludes a difference $c$ between base
$\P$-positions, and the differences $a,b$ are already excluded.
Every position outside $W$ retains a legal $a$- or $b$-move into $W$.
Thus the three-move recurrence and Lemma~\ref{lem:unique} give
$\W_0(S)=W$.

\emph{Closing at $c+a$.}
The first-block calculation gives
\[
  H:=\W_0(S)\cap[0,c+a)=W\cap[0,c)
    =T_{c+a}\cap[0,c+a),
\]
where the last equality is Proposition~\ref{prop:sieveform}.
The complementary differences $P-a=c$ and $P-c=a$ are already
excluded inside this actual initial block. Only $P-b=c-\delta$ remains.
Suppose $y=x+c-\delta$ with $x,y\in H$. Since $y<c$, we have
$0\le x<\delta$, and reducing in the base period gives
\[
  x\in\mathcal Z\cap[0,\delta),\qquad
  x+\rho-\delta\in\mathcal Z\pmod{a+b}.
\]
Thus $x\in\mathcal Z\cap[0,\delta)\cap
(\mathcal Z+(\delta-\rho))$, contrary to \eqref{eq:T1}.
The repetition argument proves $\W_0(S)=H+(c+a)\N=T_{c+a}$.

\emph{Closing at $c+b$: first identify the block.}
By Lemma~\ref{lem:T2}, both \textup{(K1)} and \textup{(K2)} hold.
The first-block formula and \textup{(K2)} give
\[
  Y(c+u)=z(c+u)(1-z(u))\qquad(0\le u<b).
\]
Indeed, when $z(c+u)=1$, the base recurrence gives $z(c+u-b)=0$.
When $z(c+u)=0$ and $z(u)=0$, one has $u\ge a$, and
\textup{(K2)} gives $z(c+u-b)=1$. In the remaining case $z(u)=1$,
both sides of the display vanish. Consequently
\[
  \begin{aligned}
  H&:=\W_0(S)\cap[0,c+b)\\
   &=(W\cap[0,c))\sqcup\{c+u:\ u\in D,\ c+u\in W\}\\
   &=T_{c+b}\cap[0,c+b).
  \end{aligned}
\]
This establishes agreement with the actual initial block before any
periodicity is asserted.

\emph{Now check the join.}
For $P=c+b$, the complementary differences $P-b=c$ and $P-c=b$
are already excluded inside $H$. A pair at the remaining distance
$P-a=c+\delta$ has the form $r,r+c+\delta$, with $0\le r<a$.
The first point is always in $H$. By the block formula, the second
belongs to $H$ exactly when
\[
  (r+\rho+\delta)\bmod(a+b)\in\mathcal Z,\qquad
  r+\delta\notin\mathcal Z.
\]
These conditions contradict \textup{(K1)}. The repetition argument
therefore gives $\W_0(S)=H+(c+b)\N=T_{c+b}$.

In every case the identified set is $P$-periodic from $0$.
\end{proof}

\begin{lemma}\label{lem:ca-angles}
For every base $(a,b)$, the angles $(\delta-a)\bmod(a+b)$ and $b$ lie in $\Th_{c+a}$.
\end{lemma}

\begin{proof}
For these two angles, $\delta-\rho\equiv a$ or $-a$ modulo $a+b$.
Either translate of $\mathcal Z$ is disjoint from $\mathcal Z$ by
Remark~\ref{rem:Zfacts}, so \eqref{eq:T1} holds.
\end{proof}

In the inert case the condition is also necessary, and the whole nim-sequence is
unchanged; this stronger statement is used in \S\ref{sec:necessity} and
\S\ref{sec:lfunction}.

\begin{prop}\label{prop:inert-iff}
$\W_0(S)=\W_0(B)$ if and only if $\rho\in\Th_{a+b}$; and in that case
$\G_S=\G_B$, so the whole nim-sequence is that of the base game: purely periodic,
with $a+b$ as a period. \textup{(}The standing assumptions $\gcd(a,b,c)=1$ and
$a\ge2$ are not used in this proposition or its proof.\textup{)}
\end{prop}

\begin{proof}
($\Leftarrow$) is proved by the argument in the case $P=a+b$ of
Theorem~\ref{thm:verify}; that argument uses neither $\gcd(a,b,c)=1$ nor $a\ge2$.
($\Rightarrow$) Suppose $\rho\notin\Th_{a+b}$, so that $z'\equiv z+\rho$ for some
$z,z'\in\mathcal Z$. These lift to $x\equiv z$ with $x$ large and $y=x+c\equiv z'$,
both in $\W_0(B)$ by $(a+b)$-periodicity; if $\W_0(S)=\W_0(B)$, the three-move recurrence
at $y$ \textup{(}Lemma~\ref{lem:unique}\textup{)} forbids $x\in\W_0(S)$, a
contradiction.

For the values, assume $\W_0(S)=\W_0(B)$ and argue by induction on $x$. If $x<c$ then
$x$ has the same options in $S$ as in $B$, and the induction hypothesis gives
$\G_S(x)=\G_B(x)$. If $x\ge c$ then the option values of $x$ in $S$ are those in $B$
together with $\G_B(x-c)$; adjoining one element to a set leaves its mex unchanged
if and only if that element differs from the mex, so it suffices to show
$\G_B(x-c)\neq\G_B(x)$. Suppose $\G_B(x-c)=\G_B(x)=i\in\{0,1,2\}$. For $i=0$ the
positions $x,x-c$ both lie in $\W_0(B)$, contradicting the first part. For $i=1$,
Lemma~\ref{lem:ferguson} gives $x-a,\,x-c-a\in\W_0(B)$, again two elements of
$\W_0(B)$ at distance $c$. For $i=2$, Lemma~\ref{lem:descent} gives
$x-2a,\,x-c-2a\in\W_0(B)$, the same contradiction. Finally $\G_B$ is purely periodic
with period $a+b$ by Proposition~\ref{prop:base}, so the same holds for $\G_S$.
\end{proof}

We now evaluate the angle tests as explicit interval conditions.
For $c+a$, this gives the three cases below. For $c+b$, we compute
the forbidden sets $\mathcal Z-\Lambda$ and $\Lambda-D$, then take
the complement of their union. These calculations depend only on the
base $(a,b)$ and require no primitivity assumption.

\begin{prop}\label{prop:theta-ca-small}
Let $b\le2a$. Then $\Th_{c+a}=\{0\}\cup[b,\,a+b)$ exactly.
\end{prop}

\begin{proof}
Here $\mathcal Z=[0,a)$ and $1\le\delta\le a$, including $b=2a$.
With $\sigma=(\delta-\rho)\bmod(a+b)$, condition \eqref{eq:T1} says
that the length-$a$ arc starting at $\sigma$ misses $[0,\delta)$.
A wrapping arc meets this interval, and a nonwrapping arc misses it
exactly when $\delta\le\sigma\le b$.
Hence $\rho\in\{0\}\cup[b,a+b)$, as claimed.
\end{proof}

\begin{thm}[arc decomposition]\label{thm:arcs}
Let $a\ge2$. Then
\[
  \Th_{c+a}=\begin{cases}
    \{0\}\cup[b,\,a+b), & b\le2a,\\
    \mathcal I, & b=ka,\ k\ \text{odd},\\
    \{\delta-a,\ b\}, & \text{otherwise}.
  \end{cases}
\]
\end{thm}

\begin{proof}
The case $b\le2a$ is Proposition~\ref{prop:theta-ca-small}; the other two cases are
window computations, carried out in Appendix~\ref{app:intervals} after
Lemma~\ref{lem:windows}.
\end{proof}

\begin{remark}\label{rem:arcs}
Lemma~\ref{lem:ca-angles} gives the inclusion $\supseteq$ in the last
case; the odd harmonic inclusion follows from the $2a$-periodic base
pattern. Appendix~\ref{app:intervals} proves the reverse inclusions
using $\Omega$. Together with Lemma~\ref{lem:idiff} and
Theorem~\ref{thm:thetacb}, this determines all admissible sets outside
the even harmonic case, where Proposition~\ref{prop:evenharm} confines
$\Th_{c+b}$ to multiples of $a$ and hence excludes primitive angles.
All are arc calculations of Proposition~\ref{prop:instances}.
\end{remark}

\begin{prop}[The two difference sets]\label{prop:arcdiff}
Assume $a\nmid b$, so that $\varepsilon\ge1$, and put $t=\lfloor\eta/2\rfloor$, so that
$\eta=2t$ or $\eta=2t+1$ according to its parity. Then, in $\Z/(a+b)\Z$:
\textup{(i)}
\[
  \mathcal Z-\Lambda=\begin{cases}
    \displaystyle\bigcup_{i=0}^{t}\big[(2i+1)a+1,\ (2i+2)a+\varepsilon-1\big], & \eta\ \text{even},\\[10pt]
    \displaystyle[1,\,a-1]\ \cup\bigcup_{i=-t-1}^{-1}\big[2ia+1,\ (2i+2)a-\varepsilon-1\big], & \eta\ \text{odd},
  \end{cases}
\]
\textup{(ii)}
\[
  \Lambda-D=\begin{cases}
    \displaystyle[1-\varepsilon,\ \varepsilon-1]\ \cup\bigcup_{i=0}^{t-1}\big[(2i+1)a+1,\ (2i+2)a+\varepsilon-1\big],
      & \eta\ \text{even},\\[10pt]
    \displaystyle\bigcup_{i=-1}^{t-1}\big[(2i+1)a+\varepsilon+1,\ (2i+3)a-1\big], & \eta\ \text{odd}.
  \end{cases}
\]
\end{prop}

\begin{proof}
The endpoints of the arcs are computed in Appendix~\ref{app:intervals}.
\end{proof}

\begin{thm}[The admissible angles for the period $c+b$]\label{thm:thetacb}
Assume $a\nmid b$ and let
\[
  C_{a,b}:=\begin{cases}
   \displaystyle[\varepsilon,\,a)\ \cup\bigcup_{i=1}^{t}\big[2ia+\varepsilon,\,(2i+1)a\big], & \eta=2t,\\[12pt]
   \displaystyle[a{+}1,\,a{+}\varepsilon]\cup\bigcup_{j=2}^{t+1}\big[(2j{-}1)a,\,(2j{-}1)a{+}\varepsilon\big]
     \cup\big[(\eta{+}2)a,\,(\eta{+}1)a{+}2\varepsilon\big], & \eta=2t+1,
  \end{cases}
\]
where empty intervals are omitted; the last interval in the odd case is nonempty precisely
when $a\le2\varepsilon$. Then
\[
  \Th_{c+b}=C_{a,b}\ \sqcup\ \{a\},
\]
and consequently
\[
  \Th_{c+b}\setminus\Th_{a+b}=C_{a,b}.
\]
In particular $C_{a,b}\subseteq\Th_{c+b}$; for $\eta=0$ the first display reads
$C_{a,b}=[\delta,a)$. If instead $a\mid b$, say $b=ka$ with $k$ odd \textup{(}we call the shapes with
$a\mid b$ \emph{harmonic}\textup{)}, then $\varepsilon=0$, the arc
$\Lambda$ of \eqref{eq:I1Lambda} is empty, both conditions of \eqref{eq:T2} are vacuous,
and $\Th_{c+b}=\Z/(a+b)\Z$.
\end{thm}

\begin{proof}
Read off the gaps of the union $(\mathcal Z-\Lambda)\cup(\Lambda-D)$ of
Proposition~\ref{prop:arcdiff}; the reading is carried out in Appendix~\ref{app:intervals}.
\end{proof}

\subsection*{Reading the completed block}

The preceding argument determines the initial $\P$-set and shows when it
repeats. We now read that set in increasing order. Each complete base block
contributes $N$ positions, and the final partial block contributes
$C_a(\min(\rho,b))$. Put
\[
  N_c:=K\,N+C_a(\min(\rho,b))
      =K\,N+\#\big(\mathcal Z\cap[0,\rho)\big),
\]
\[
  D_\rho:=\{u\in D:\ (\rho+u)\bmod(a+b)\in\mathcal Z\}=\{\tau_0<\tau_1<\cdots\};
\]
Thus $N_c$ counts the base $\P$-positions below $c$, and $D_\rho$
is the \emph{band} of offsets used in the $c+b$ construction. For a candidate
period $P\in\{c+a,\,c+b\}$ put $E:=\varnothing$ if $P=c+a$ and $E:=D_\rho$ if
$P=c+b$.

\begin{thm}\label{thm:cut}
Let $S=\{a,b,c\}$ be non-additive, let $P\in\{c+a,\,c+b\}$, and suppose
$\W_0(S)=T_P$, the $c$-sieve of Definition~\ref{def:sieve}. Then
\[
  \W_0(S)\cap[0,P)\ =\ \{w_B(n):0\le n<N_c\}\ \sqcup\ \big(c+E\big).
\]
Put $M:=N_c+|E|$. The increasing list of this block is
\[
  v(t)=\begin{cases}
    w_2(t), & 0\le t<N_c,\\[2pt]
    c+\tau_{t-N_c}, & N_c\le t<M,
  \end{cases}
\]
where $\tau_0<\cdots<\tau_{|E|-1}$ enumerate $E$; the second case is empty
when $E=\varnothing$. For every $n\in\N$, the $n$-th $\P$-position is
\[
  w_3(n):=w(n)=P\Big\lfloor\frac nM\Big\rfloor+v(n\bmod M).
\]
In the inert situation $\W_0(S)=\W_0(B)$ the corresponding statement is $w(n)=w_B(n)$, with
no cut.
\end{thm}

\begin{proof}
Proposition~\ref{prop:sieveform} gives the part below $c$ and, for
$P=c+b$, the additional band $W\cap(c+D)=c+D_\rho$.
The prefix below $c=K(a+b)+\rho$ consists of $K$ full base periods
and the part indexed by $\mathcal Z\cap[0,\rho)$, so it contains
$KN+\#(\mathcal Z\cap[0,\rho))=N_c$ points.
The band lies in $[c,P)$, above all these points, so its increasing
enumeration follows the first $N_c$ values of $w_2$, giving $v$.
Since $N_c>0$, we may write $n=kM+t$, $0\le t<M$; periodicity then
places the $n$-th point at $kP+v(t)$.
When $E=\varnothing$ only the prefix remains, and the inert case has no cut.
\end{proof}

\begin{remark}\label{rem:ergodic}
The band $D_\rho$ consists of base gaps that the rotation carries into
$\mathcal Z$. Figure~\ref{fig:mechanism} shows how it completes the cut
block for $P=c+b$, using the base of Example~\ref{ex:38}.

\end{remark}

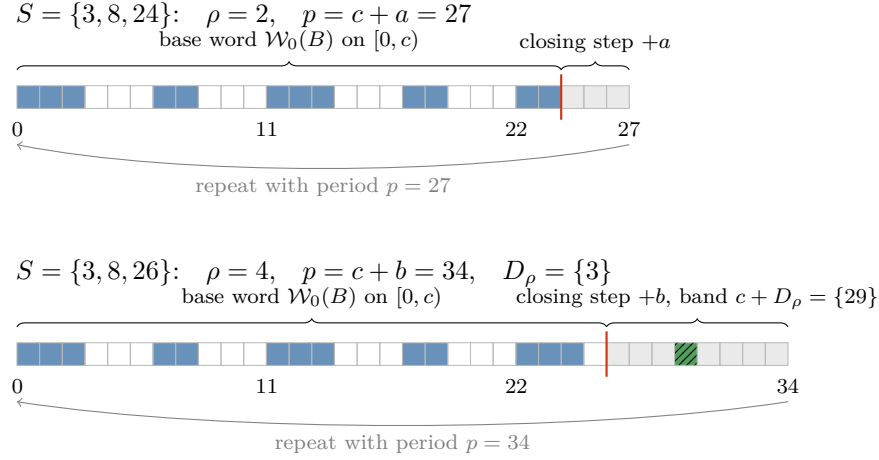
\begin{figure}[htbp]
\centering
\begin{tikzpicture}[x=0.30cm,y=0.30cm]
\begin{scope}
\node[anchor=west,font=\small] at (-0.5,4.1) {$S=\{3,8,24\}$:\quad $\rho=2$,\quad $p=c+a=27$};
\fill[preC!35] (24,0) rectangle (27,1);
\fill[inertC!80] (0,0) rectangle (1,1);
\fill[inertC!80] (1,0) rectangle (2,1);
\fill[inertC!80] (2,0) rectangle (3,1);
\fill[inertC!80] (6,0) rectangle (7,1);
\fill[inertC!80] (7,0) rectangle (8,1);
\fill[inertC!80] (11,0) rectangle (12,1);
\fill[inertC!80] (12,0) rectangle (13,1);
\fill[inertC!80] (13,0) rectangle (14,1);
\fill[inertC!80] (17,0) rectangle (18,1);
\fill[inertC!80] (18,0) rectangle (19,1);
\fill[inertC!80] (22,0) rectangle (23,1);
\fill[inertC!80] (23,0) rectangle (24,1);
\foreach \i in {0,...,26} \draw[gray!55,line width=0.3pt] (\i,0) rectangle (\i+1,1);
\draw[cutC,thick] (24,-0.5) -- (24,1.5);
\draw[decorate,decoration={brace,amplitude=3pt}] (0,1.7) -- (24,1.7)
  node[midway,above=3pt,font=\scriptsize] {base word $\W_0(B)$ on $[0,c)$};
\draw[decorate,decoration={brace,amplitude=3pt}] (24,1.7) -- (27,1.7)
  node[midway,above=3pt,font=\scriptsize] {closing step $+a$};
\foreach \x in {0,11,22} \node[font=\scriptsize,anchor=north] at (\x+0.02,-0.15) {$\x$};
\node[font=\scriptsize,anchor=north] at (27,-0.15) {$27$};
\draw[->,gray] (27,-1.6) .. controls (20,-3.0) and (7,-3.0) .. (0,-1.6);
\node[font=\scriptsize,gray,anchor=north] at (13.5,-2.6) {repeat with period $p=27$};
\end{scope}
\begin{scope}[yshift=-3.4cm]
\node[anchor=west,font=\small] at (-0.5,4.1) {$S=\{3,8,26\}$:\quad $\rho=4$,\quad $p=c+b=34$,\quad $D_\rho=\{3\}$};
\fill[preC!35] (26,0) rectangle (34,1);
\fill[inertC!80] (0,0) rectangle (1,1);
\fill[inertC!80] (1,0) rectangle (2,1);
\fill[inertC!80] (2,0) rectangle (3,1);
\fill[inertC!80] (6,0) rectangle (7,1);
\fill[inertC!80] (7,0) rectangle (8,1);
\fill[inertC!80] (11,0) rectangle (12,1);
\fill[inertC!80] (12,0) rectangle (13,1);
\fill[inertC!80] (13,0) rectangle (14,1);
\fill[inertC!80] (17,0) rectangle (18,1);
\fill[inertC!80] (18,0) rectangle (19,1);
\fill[inertC!80] (22,0) rectangle (23,1);
\fill[inertC!80] (23,0) rectangle (24,1);
\fill[inertC!80] (24,0) rectangle (25,1);
\fill[cbC!80] (29,0) rectangle (30,1);
\fill[pattern=north east lines,pattern color=black] (29,0) rectangle (30,1);
\foreach \i in {0,...,33} \draw[gray!55,line width=0.3pt] (\i,0) rectangle (\i+1,1);
\draw[cutC,thick] (26,-0.5) -- (26,1.5);
\draw[decorate,decoration={brace,amplitude=3pt}] (0,1.7) -- (26,1.7)
  node[midway,above=3pt,font=\scriptsize] {base word $\W_0(B)$ on $[0,c)$};
\draw[decorate,decoration={brace,amplitude=3pt}] (26,1.7) -- (34,1.7)
  node[midway,above=3pt,font=\scriptsize] {closing step $+b$, band $c+D_\rho=\{29\}$};
\foreach \x in {0,11,22} \node[font=\scriptsize,anchor=north] at (\x+0.02,-0.15) {$\x$};
\node[font=\scriptsize,anchor=north] at (34,-0.15) {$34$};
\draw[->,gray] (34,-1.6) .. controls (25,-3.0) and (9,-3.0) .. (0,-1.6);
\node[font=\scriptsize,gray,anchor=north] at (17,-2.6) {repeat with period $p=34$};
\end{scope}
\end{tikzpicture}
\caption{The two phases for the base game $B=\{3,8\}$ of Example~\ref{ex:38}; dark
cells are $\P$-positions. One period consists of the base pattern below
$c$, followed by an interval of length $a$ or $b$. This interval
contains no $\P$-positions for $p=c+a$ \textup{(}top\textup{)}, and contains the
band $c+D_\rho$ \textup{(}hatched\textup{)} for $p=c+b$ \textup{(}bottom\textup{)}.}
\label{fig:mechanism}
\end{figure}

When $b\le2a$, the admissible $c+a$ angles have $a\mid N_c$,
so the generator simplifies as follows.

\begin{corollary}[Two-term form]\label{cor:twoterm}
Let $b\le2a$ and $\rho\in\{0\}\cup[b,a+b)$. Then $\mathcal Z=[0,a)$, $N_c=Ka$ for $\rho=0$ and
$N_c=(K+1)a$ for $\rho\ge b$, and with
\[
  e:=(c+a)-\frac{(a+b)N_c}{a}
   =\begin{cases}a,&\rho=0,\\ \rho-b,&\rho\ge b,\end{cases}
\]
the $n$-th $\P$-position of $S$ is
\[
  w(n)\;=\;n+b\Big\lfloor\frac{n}{a}\Big\rfloor+e\Big\lfloor\frac{n}{N_c}\Big\rfloor,
  \qquad n\ge0 .
\]
\end{corollary}

\begin{proof}
Here $N=a$, $w_B(n)=n+b\lfloor n/a\rfloor$ and
$N_c=Ka+\min(\rho,a)$, which gives the stated counts and $e$.
Write $N_c=ad$ and $n=kN_c+aj+r$ with $0\le j<d$, $0\le r<a$.
Proposition~\ref{prop:theta-ca-small} and Theorem~\ref{thm:verify}
give $\W_0(S)=T_{c+a}$, so Theorem~\ref{thm:cut} gives
$w(n)=k(c+a)+(a+b)j+r$.
Substitution of $\lfloor n/a\rfloor=kd+j$ and
$\lfloor n/N_c\rfloor=k$ proves the formula.
\end{proof}

\emph{The other short forms.}
For $\rho=b-2a$, under the standing primitive assumptions,
$b>2a$ and $a\nmid b$. Since $C_a(b)-C_a(b-2a)=a$, the cut count is
$M=KN+C_a(b-2a)=(K+1)N-a$. Writing $n=kM+t$, $0\le t<M$, gives
\[
 w_2(n+ak)-2ak
   =k\bigl((K+1)(a+b)-2a\bigr)+w_2(t)
   =k(c+a)+w_2(t).
\]
Thus the same cut construction becomes
$w_3(n)=w_2(n+a\lfloor n/M\rfloor)-2a\lfloor n/M\rfloor$.

In the short $c+b$ branch $b-a\le\rho<a$, one has $b<2a$,
$N=a$, $N_c=Ka+\rho$ and $D_\rho=\varnothing$.
Indeed, for $u\in[a,b)$ the sum $\rho+u$ is below $a+b$
and at least $a$, so it cannot belong to $\mathcal Z=[0,a)$.
Put $M=Ka+\rho$ and write $n=kM+t$, $0\le t<M$. Then
$n+(a-\rho)k=ka(K+1)+t$, and hence
\[
 n+b\left\lfloor\frac{n+(a-\rho)\lfloor n/M\rfloor}{a}\right\rfloor
   =k(c+b)+t+b\lfloor t/a\rfloor
   =k(c+b)+w_2(t).
\]
These are the short forms in Table~\ref{tab:closedforms}.

\subsection*{Evaluating the band}\label{sec:band-enumeration}

For the remaining $c+b$ rows, we enumerate $D_\rho$ in two steps:
split it according to whether $\rho+u$ wraps around $a+b$, then
count and list the points in each resulting window. The resulting
list supplies the additional positions in Theorem~\ref{thm:cut}.
Put $m=a+b$. A band point $u$ satisfies
\[
 a\le u<b,\qquad u\bmod2a\ge a,\qquad
 (\rho+u)\bmod m\in\mathcal Z.
\]
Since $0\le\rho<m$ and $u<b<m$, the sum $\rho+u$ crosses $m$
at most once. Requiring its representative to lie below $b$ leaves
the two ordered windows
\[
 [L_0,V_0)=[a,b-\rho),\qquad
 [L_1,V_1)=[\max(a,m-\rho),b).
\]
In window $j\in\{0,1\}$, set
\[
 \sigma_j=(\rho-jm)\bmod2a,\qquad d_j=a+\max(0,a-\sigma_j),
\]
\[
 \ell_j=\min(a,2a-\sigma_j)-\max(0,a-\sigma_j).
\]
For $r=u\bmod2a\in[a,2a)$, the second residue condition is
$(r+\sigma_j)\bmod2a<a$. Its intersection with $[a,2a)$ is
$[d_j,d_j+\ell_j)$, so the selected points of this window lie in
$d_j+[0,\ell_j)+2a\Z$.

To count and list points of $d+[0,\ell)+2a\Z$, where
$d,\ell\in\Z$ and $1\le\ell\le2a$, define
\[
 R_{d,\ell}(x)=
   \ell\left\lfloor\frac{x-d}{2a}\right\rfloor
   +\min\bigl(\ell,(x-d)\bmod2a\bigr),
\qquad
 E_{d,\ell}(t)=d+t+(2a-\ell)\left\lfloor\frac t\ell\right\rfloor.
\]
Here $x,t\in\Z$; floor division and nonnegative remainders are used,
also for negative arguments. The increment $R(x+1)-R(x)$ is $1$
exactly when $x$ is selected. Therefore the number of selected points
in $[L,V)$ is $R(V)-R(L)$. Writing $t=q\ell+r$, $0\le r<\ell$,
gives $E(t)=d+2aq+r$ and $R(E(t))=t$. The $s$-th selected point
in the window is consequently $E(R(L)+s)$.

Apply these formulas to the two windows:
\[
 T_j=R_{d_j,\ell_j}(V_j)-R_{d_j,\ell_j}(L_j),\qquad
 e_j(s)=E_{d_j,\ell_j}\bigl(R_{d_j,\ell_j}(L_j)+s\bigr)
 \quad(0\le s<T_j).
\]
If $V_j\le L_j$ or $\ell_j=0$, set $T_j=0$ and do not evaluate $e_j$.
Since $L_1>V_0$, the first window precedes the second. Thus
\[
 M=N_c+T_0+T_1,\qquad
 v(t)=
 \begin{cases}
 w_2(t),&0\le t<N_c,\\
 c+e_0(t-N_c),&N_c\le t<N_c+T_0,\\
 c+e_1(t-N_c-T_0),&N_c+T_0\le t<M.
 \end{cases}
\]
Together with $w_3(n)=(c+b)\lfloor n/M\rfloor+v(n\bmod M)$,
this evaluates every entry of the general $c+b$ row without a search
or a sorting step. The window formulas enumerate the band for any
angle; identifying the repeated block with the game's $\P$-positions
requires $\W_0(S)=T_{c+b}$, as assumed in Theorem~\ref{thm:cut}.

\needspace{10\baselineskip}
\subsection*{The Main Theorem}

\begin{thm}[Main Theorem]\label{thm:main}
Let $a\ge2$, $\gcd(a,b,c)=1$, $c>b$ and $c\neq a+b$, and let $\rho=c\bmod(a+b)$ be the angle.
Suppose $\rho$ is admissible for at least one candidate period
\textup{(}Definition~\ref{def:theta}\textup{)},
\[
  \rho\ \in\ \Th_{a+b}\cup\Th_{c+a}\cup\Th_{c+b},
\]
and let $P$ be the least admissible candidate. Then $\G_S$ is purely periodic, its least
period is $P$, and
\[
  \W_0(S)=\begin{cases}
    \W_0(B), & \rho\in\Th_{a+b},\\[2pt]
    T_P, & \text{otherwise}.
  \end{cases}
\]
\end{thm}

The explicit formulas are collected in Table~\ref{tab:closedforms}.
The angle sets are evaluated by Lemma~\ref{lem:idiff} and
Theorems~\ref{thm:thetacb}, \ref{thm:arcs}; Proposition~\ref{prop:evenharm}
excludes primitive $c+b$ angles in the even harmonic case.

\begin{proof}
\emph{Inert case $\rho\in\Th_{a+b}$.} Proposition~\ref{prop:inert-iff} gives
$\G_S=\G_B$, so $q=0$ and the least period is that of $\W_0(B)$. By
\cite[Theorem 3.4]{MiklosPost} the latter is $a+b$ unless $b$ is an odd multiple of $a$;
in that harmonic case $\mathcal I$ consists of multiples of $a$
\textup{(}Lemma~\ref{lem:idiff}\textup{)}, so $a\mid\rho$ would force $a\mid c$,
contradicting $\gcd(a,b,c)=1$ with $a\ge2$. Hence the least period is $a+b=P$.

\emph{Phase case $\rho\notin\Th_{a+b}$.}
Theorem~\ref{thm:verify}, applied to the least admissible candidate
$P\in\{c+a,c+b\}$, gives $\W_0(S)=T_P$ and hence
$P$-periodicity of $\W_0(S)$. The $\P$-set description and the
admissibility of $P$ give hypothesis \textup{(H)} of
\S\ref{sec:lfunction}.
Corollary~\ref{cor:grundy-period} transfers this period to $\G_S$.
Lemma~\ref{lem:sameperiod} identifies the two least periods, and
Proposition~\ref{prop:minimality} shows that they equal $P$.
These results use the $\P$-set description and admissibility
conditions established here (Remark~\ref{rem:deferred}).
\end{proof}

Table~\ref{tab:closedforms} lists the conditions, the $n$-th
$\P$-position and the least period in each covered branch.
The uncovered angles and the known cases of necessity are discussed in
\S\ref{sec:necessity}.

\begin{remark}\label{rem:a1-fails}
The least-period assertion of the Main Theorem does not extend
unchanged to $a=1$. For $S=\{1,3,5\}$ one has $a+b=4$,
$\mathcal Z=\{0,2\}$ and $\rho=1\in\mathcal I=\{1,3\}$, so the Main Theorem would return
the least period $a+b=4$; but every move of $S$ is odd, so $\G_S(x)=x\bmod2$ by induction
on $x$, and the least period is $2$. The harmonic exclusion fails because $a=1$ makes $a\mid c$ vacuous.
The known purely periodic formulas for $a=1$ are recorded in
Table~\ref{tab:a1-known}; see Remark~\ref{rem:a1} for their attribution.

\end{remark}

\needspace{15\baselineskip}
\begingroup
\small
\setlength{\LTleft}{0pt}
\setlength{\LTright}{\fill}
\setlength{\LTcapwidth}{\textwidth}
\setlength{\tabcolsep}{4pt}
\renewcommand{\arraystretch}{1.35}
\begin{longtable}{@{}p{.24\textwidth}p{.57\textwidth}p{.14\textwidth}@{}}
\caption{Conditions, increasing $\P$-position formulas, and least periods
for the covered three-move games.}\label{tab:closedforms}\\
\multicolumn{3}{@{}p{\textwidth}@{}}{%
Every non-additive row assumes $2\le a<b<c$, $\gcd(a,b,c)=1$, and
$c\ne a+b$. The additive row has its own stated hypotheses.
The index is $n=0,1,\ldots$, with $w_3(0)=0$; $p$ is the least period
of both the $\P$-set and the full nim-sequence.}\\[5pt]
\multicolumn{3}{@{}p{\textwidth}@{}}{%
We use $N=C_a(b)$, $K=\lfloor c/(a+b)\rfloor$,
$\rho=c\bmod(a+b)$, $\sigma=\rho\bmod2a$, and
$\varepsilon=b\bmod a$. This is the same $\varepsilon$ as in
$b-a=\eta a+\varepsilon$, and $\lfloor b/a\rfloor=\eta+1$.}\\[6pt]
\toprule
Conditions & $n$-th $\P$-position & \shortstack[l]{Least period\\$p$}\\
\midrule
\endfirsthead
\caption[]{Conditions, $\P$-position formulas, and least periods
\textup{(}continued\textup{)}.}\\
\multicolumn{3}{@{}p{\textwidth}@{}}{%
The non-additive rows assume $2\le a<b<c$, $\gcd(a,b,c)=1$,
and $c\ne a+b$. The notation and index convention are as above.}\\[5pt]
\toprule
Conditions & $n$-th $\P$-position & \shortstack[l]{Least period\\$p$}\\
\midrule
\endhead
\midrule
\multicolumn{3}{r@{}}{Continued on the next page}\\
\endfoot
\bottomrule
\endlastfoot

\textit{Inert}\newline
$a<b\le2a$,\newline $a\le\rho\le b$,\newline
or\newline
$b>2a$, $a\nmid b$,\newline $\rho\in\{a,b\}$
& $w_3(n)=w_2(n)$
& $a+b$\\[9pt]
\midrule

$b\le2a$,\newline $\rho=0$ or $\rho>b$
& $\begin{aligned}[t]
   w_3(n)&=n+b\Big\lfloor\frac na\Big\rfloor\\
   &\quad+\Big(c+a-\frac{(a+b)M}{a}\Big)
                  \Big\lfloor\frac nM\Big\rfloor,\\[2pt]
   M&=a\Big\lceil\frac{c}{a+b}\Big\rceil
   \end{aligned}$
& $c+a$\\[10pt]

$b>2a$, $a\nmid b$,\newline $\rho=b-2a$
& $\begin{aligned}[t]
   w_3(n)&=w_2\!\Big(n+a\Big\lfloor\frac nM\Big\rfloor\Big)
                  -2a\Big\lfloor\frac nM\Big\rfloor,\\[3pt]
   M&=(K+1)N-a
   \end{aligned}$
& $c+a$\\[10pt]
\midrule

$a<b<2a$,\newline $b-a\le\rho<a$
& $\begin{aligned}[t]
   w_3(n)&=n+b\Big\lfloor
                  \frac{n+(a-\rho)\lfloor n/M\rfloor}{a}
                  \Big\rfloor,\\[3pt]
   M&=Ka+\rho
   \end{aligned}$
& $c+b$\\[10pt]

$b>2a$, $\varepsilon>0$,\newline
$\lfloor b/a\rfloor$ odd,\newline
$\varepsilon\le\sigma\le a$,\newline $\rho\ne a$
& $\displaystyle
   w_3(n)=(c+b)\Big\lfloor\frac nM\Big\rfloor+v(n\bmod M)$
& $c+b$\\[9pt]

$b>2a$, $\varepsilon>0$,\newline
$\lfloor b/a\rfloor$ even,\newline
$a\le\sigma\le a+\varepsilon$,\newline
$\rho\le b+\varepsilon$, $\rho\ne a$
& $\displaystyle
   w_3(n)=(c+b)\Big\lfloor\frac nM\Big\rfloor+v(n\bmod M)$
& $c+b$\\[9pt]

$b/a\in\{3,5,7,\ldots\}$,\newline
no additional condition\newline on $\rho$
& $\displaystyle
   w_3(n)=(c+b)\Big\lfloor\frac nM\Big\rfloor+v(n\bmod M)$
& $c+b$\\[9pt]
\midrule

\textit{Additive}\newline
$c=a+b$,\newline
$2\le a<b$,\newline $\gcd(a,b)=1$
& $\displaystyle
   w_3(n)=n+a\Big\lfloor\frac na\Big\rfloor
             +b\Big\lfloor\frac nN\Big\rfloor$
& $\displaystyle\frac{a(2N+b)}{\gcd(a,N)}$\\[8pt]
\midrule
\multicolumn{3}{@{}p{\textwidth}@{}}{%
The two-move generator $w_2$ uses the same $a,b$ as its row.
Where $M$ occurs, it is the number of $\P$-positions in one least period.
In the three rows using $v$, the increasing list
$v(0)<\cdots<v(M-1)$ and its length $M$ are given explicitly by the
two-window construction in \S\ref{sec:band-enumeration}.}\\[5pt]
\multicolumn{3}{@{}p{\textwidth}@{}}{%
Under the stated hypotheses the rows are disjoint. Their conditions
select among the proved covered constructions; they do not decide
pure periodicity for non-additive rulesets outside the table.
The necessity assertion remains Conjecture~\ref{conj:necessity}.}\\
\end{longtable}
\endgroup

\newpage
\subsection*{The known case with smallest move 1}

For comparison with Table~\ref{tab:closedforms}, we collect the known
purely periodic games $S=\{1,b,c\}$ in Table~\ref{tab:a1-known}.
The classification and the full nim-sequences are given in
\cite[Example~2.4 and Proposition~2.5]{Zhang}; the odd-$b$, even-$c$
formula is due to Ho \cite[Theorem~4]{Ho}.
Only the purely periodic cases are listed. Every omitted case has
positive pre-period by the same classification.

\begingroup
\small
\setlength{\LTleft}{0pt}
\setlength{\LTright}{\fill}
\setlength{\LTcapwidth}{\textwidth}
\setlength{\tabcolsep}{4pt}
\renewcommand{\arraystretch}{1.45}
\begin{longtable}{@{}p{.18\textwidth}p{.335\textwidth}p{.345\textwidth}p{.08\textwidth}@{}}
\caption{Known purely periodic cases with $a=1$. The displayed finite
word repeats from the first term; $p$ is its least period.}
\label{tab:a1-known}\\
\multicolumn{4}{@{}p{\textwidth}@{}}{%
Assume $1<b<c$ and write $c=K(b+1)+\rho$, $0\le\rho\le b$.
For even $b$, put $k=b/2$ and $H_b=(01)^k2$.
Word powers denote concatenation. The index is $n=0,1,\ldots$.}\\[5pt]
\toprule
Conditions & $n$-th $\P$-position $w_3(n)$ & Grundy word & Least $p$\\
\midrule
\endfirsthead
\caption[]{Known purely periodic cases with $a=1$ \textup{(}continued\textup{)}.}\\
\toprule
Conditions & $n$-th $\P$-position $w_3(n)$ & Grundy word & Least $p$\\
\midrule
\endhead
\bottomrule
\endfoot
\raggedright $b,c$ odd
& $2n$
& $01$
& $2$\\[5pt]
\raggedright $b$ odd, $c$ even
& $2n+b\lfloor 2n/c\rfloor$
& $(01)^{c/2}(23)^{(b-1)/2}2$
& $b+c$\\[5pt]
\raggedright $b$ even,\newline $\rho\in\{1,b\}$
& $2n+\lfloor n/k\rfloor$
& $H_b$
& $b+1$\\[5pt]
\raggedright $b\ge4$ even,\newline $3\le\rho\le b-1$,\newline $\rho$ odd
& $\begin{gathered}
2n+\lfloor n/k\rfloor\\[-1pt]
{}+(\rho-1)\lfloor n/M\rfloor,\\[-1pt]
M=(K+1)k
\end{gathered}$
& $H_b^{K+1}(32)^{(\rho-1)/2}$
& $b+c$\\[5pt]
\raggedright $b\ge4$ even,\newline $\rho=b-2$
& $\begin{gathered}
2n+\left\lfloor\dfrac{n+\lfloor n/M\rfloor}{k}\right\rfloor,\\[2pt]
M=(K+1)k-1
\end{gathered}$
& $H_b^K(01)^{k-1}2$
& $c+1$\\[5pt]
\raggedright $b$ even,\newline $c=b+1$
& $2n+b\lfloor n/k\rfloor$
& $(01)^k(23)^k$
& $2b$\\[5pt]
\raggedright $b=2$, $\rho=0$,\newline $c>3$
& $3n+\lfloor n/K\rfloor$
& $(012)^K3$
& $c+1$\\
\end{longtable}
\endgroup

The $\P$-position formulas follow by reading the zeros of these known
words. For even $b$, the base enumeration is
$w_B(n)=2n+\lfloor n/k\rfloor$.
The even-$b$ rows with $b\ge4$ and periods $b+c$ and $c+1$ repeat its first
$(K+1)k$ and $(K+1)k-1$ terms, respectively, with the indicated
periods, by the same block enumeration used in
Theorem~\ref{thm:cut}.
For $c=b+1$, both parities give
\[
 N=\lceil b/2\rceil,\qquad
 w_3(n)=2n+b\lfloor n/N\rfloor,\qquad p=b+2N.
\]
These are also the $a=1$ specializations of the additive formulas in
Theorem~\ref{thm:additive}; the proof there retains $a\ge2$.
The least-period rule of the Main Theorem cannot be specialized in
the same way: if $b,c$ are odd, the least period is $2$, smaller than
every candidate sum (Remark~\ref{rem:a1-fails}).

\needspace{22\baselineskip}

\section{The additive case}\label{sec:additive}

On the additive line $c=a+b$, the primitive hypothesis is $\gcd(a,b)=1$;
we retain $a\ge2$. The first $N$ $\P$-positions agree with the base game.
At $c$, the base pattern would return to a $\P$-position, but the new move
reaches $0$. The continuation is obtained from the one-move $\P$-sequence
by adding $jb$ to each term whose index lies in $[jN,(j+1)N)$,
for $j=0,1,\ldots$. The following theorem verifies this construction
and gives a unified form of the classical additive parametrizations
(Remark~\ref{rem:pq}).

\needspace{14\baselineskip}
\begin{thm}\label{thm:additive}
Let $c=a+b$ with $a\ge2$ and $\gcd(a,b)=1$. Put
\[
  N=a\Big\lfloor\frac{b}{2a}\Big\rfloor+\min(a,b\bmod2a),\qquad
  Q=\operatorname{lcm}(a,N).
\]
The $n$-th $\P$-position of $S$, for $n\ge0$, is
\[
  w_3(n):=w(n)=n+a\Big\lfloor\frac{n}{a}\Big\rfloor
          +b\Big\lfloor\frac{n}{N}\Big\rfloor.
\]
Both $\W_0(S)$ and $\G_S$ are purely periodic, with least period
\[
  p_{\mathrm{add}}=2Q+b\frac{Q}{N}
                 =\frac{a(2N+b)}{\gcd(a,N)}.
\]
Write $\alpha_n=w(n)$ for $0\le n<Q$, so that
$\W_0(S)\cap[0,p_{\mathrm{add}})=\{\alpha_n:0\le n<Q\}$.
In terms of $\delta=b-a=\eta a+\varepsilon$, the parameters specialize to
\[
\begin{array}{c|c|c|c}
 &N&Q&p_{\mathrm{add}}\\ \hline
\eta\text{ even}&(\tfrac\eta2+1)a&N&(\eta+2)a+b\\[2pt]
\eta\text{ odd}&m_\eta=\delta-\tfrac{\eta-1}{2}a&a\,m_\eta&a(2m_\eta+b).
\end{array}
\]
For even $\eta$, $\alpha_n=n+a\lfloor n/a\rfloor$ on $0\le n<Q=N$;
for odd $\eta$, the formula for $w(n)$ uses $N=m_\eta$.
\end{thm}

\begin{proof}
We first prove the formula for $\W_0$ and its least period. Let
$u(n)=n+a\lfloor n/a\rfloor$, and put
\[
  \gamma=2N-b,\qquad p=2Q+bQ/N,\qquad
  W_{\mathrm a}=\{u(n)+b\lfloor n/N\rfloor:0\le n<Q\}.
\]
The definition of $N$, together with $\gcd(a,b)=1$ and $a\ge2$, gives
$0<\gamma<a$ and $N\bmod a\in\{0,\gamma\}$. The proposed enumeration is
strictly increasing and satisfies
\[
  w(n+Q)=w(n)+p,\qquad \max W_{\mathrm a}=w(Q-1)=p-a-b-1.
\]
Thus $W_{\mathrm a}$ and its translates by $a,b,c$ lie in $[0,p)$.

\smallskip\emph{No move joins two proposed $\P$-positions.}
Within a group $[jN,(j+1)N)$ of indices, the term $bj$ cancels from a
difference. The one-move $\P$-set has no difference $a$. If $N=ha$, the
group's diameter in $u$ is $2N-a-1<b$, so neither $b$ nor $c$ occurs.
If $N=ha+\gamma$, write an index difference as $d=qa+r$, $0\le r<a$.
Lemma~\ref{lem:diffspec} gives
\[
  u(n)-u(n-d)=2qa+r+a\chi,\qquad
  \chi=[\,n\bmod a<r\,]\in\{0,1\}.
\]
Since $b=2ha+\gamma$ and $c=(2h+1)a+\gamma$, either move difference
would require $r=\gamma$ and $q=h$, hence $d=N$. This is impossible
within one group.

If two indices lie in different groups, the corresponding values of
$w$ differ by more than $b$. If their group numbers differ by at least
two, the difference exceeds $2b>c$. For adjacent groups, a difference
$c=a+b$ would require the corresponding values of $u$ to differ by
$a$, which is impossible. Thus no move difference occurs. A move difference
modulo $p$ that is not a literal move difference would have absolute
value $p-s\ge p-c$ for some $s\in\{a,b,c\}$, exceeding the diameter
of $W_{\mathrm a}$. Repeating $W_{\mathrm a}$ with step $p$ therefore
creates no collision at a join.

\smallskip\emph{Every other position has a move into the proposed set.}
The translates $W_{\mathrm a}+a$, $W_{\mathrm a}+b$, $W_{\mathrm a}+c$
lie in $[0,p)\setminus W_{\mathrm a}$. Their only possible overlap is
$(W_{\mathrm a}+a)\cap(W_{\mathrm a}+b)$: the other two pairs would
require a difference $b$ or $a$ between points of $W_{\mathrm a}$.
The overlap is counted by pairs of points at distance $\delta=b-a$.
Such a pair lies within one index group, since $\delta<b$.
For $n'<n$ in $[jN,(j+1)N)$, the complete list of possibilities is
\[
\begin{array}{c|c|l}
N\bmod a&n-n'&\text{conditions on }n'\\ \hline
0&N-\gamma&jN\le n'<jN+\gamma\\[2pt]
\gamma&N-a&jN\le n'<jN+a,\quad n'\bmod a\ge a-\gamma.
\end{array}
\]
To obtain the rows, substitute $\delta=(2h-2)a+(a-\gamma)$ for $N=ha$
and $\delta=(2h-1)a+\gamma$ for $N=ha+\gamma$ into the index form.
The respective carries are $\chi=0$ and $\chi=1$; the latter is equivalent
to $n'\bmod a\ge a-\gamma$. Requiring both indices to remain in the group
gives the stated bounds. Each row has $\gamma$ pairs: in the second,
the $a$ possible starting indices contain each residue once. Thus
\[
  N_\delta:=\#\{w\in W_{\mathrm a}:w+\delta\in W_{\mathrm a}\}
           =\gamma Q/N=4Q-p
           =\begin{cases}a-\varepsilon,&\eta\text{ even},\\
                         a\varepsilon,&\eta\text{ odd}.
             \end{cases}
\]
Here $\gamma=a-\varepsilon$ for even $\eta$, and $\gamma=\varepsilon$
for odd $\eta$, using the parameter table above. There is no triple
overlap, so inclusion--exclusion gives
\[
  |(W_{\mathrm a}+a)\cup(W_{\mathrm a}+b)\cup(W_{\mathrm a}+c)|
  =3Q-\gamma Q/N=p-Q=|[0,p)\setminus W_{\mathrm a}|.
\]
Thus these translates cover the entire complement.

Let $T=W_{\mathrm a}+p\N$, which has no move into itself. For
$x=kp+y\notin T$, $0\le y<p$, the covering gives $y=w+s$ with
$w\in W_{\mathrm a}$ and $s\in\{a,b,c\}$. Since $x-s=kp+w\ge0$,
this move is legal and reaches $T$. The recurrence holds, including at the
initial boundary, so $T=\W_0(S)$ by Lemma~\ref{lem:unique}.

\smallskip\emph{The least period.}
The successive gaps are
\[
  w(n+1)-w(n)
  =1+a[\,a\mid n+1\,]+b[\,N\mid n+1\,].
\]
They have period $Q$, and their largest value $1+a+b$ occurs exactly
when $Q\mid n+1$. Hence their least period is $Q$. A spatial period
of the $\P$-set translates its increasing enumeration by a fixed
number of indices and so induces a period of the gap sequence.
The least spatial period is therefore $w(Q)-w(0)=p$.

The parameter table follows from $a\mid N$ when $\eta$ is even and
$\gcd(a,N)=\gcd(a,\varepsilon)=1$ when $\eta$ is odd. For the additive
reconstruction we retain, in the even case,
\[
  \mathcal Z_{\mathrm a}:=\bigcup_{j=0}^{\eta/2}[2ja,(2j+1)a)
  =W_{\mathrm a}.
\]

Corollary~\ref{cor:grundy-period}, using the $\W_0$-formula and
collision counts above, now gives $p$-periodicity of $\G_S$ independently
of the nim-value conclusion here (Remark~\ref{rem:deferred}). Its least
period is $p$, since the zero set already has that least period.
\end{proof}

\begin{remark}\label{rem:additive-dilate}
If $g=\gcd(a,b)>1$ and $a/g\ge2$, apply Theorem~\ref{thm:additive}
to $S'=(a/g,b/g,c/g)$ and then Lemma~\ref{lem:gcd}:
\[
  q(S)=0,\qquad p(S)=g\cdot p_{\mathrm{add}}(a/g,b/g),\qquad
  \W_0(S)=\{gm+r:\ m\in\W_0(S'),\ 0\le r<g\}.
\]
The even-$\eta$ period formula commutes with dilation; the odd-$\eta$
formula $a(2m_\eta+b)$ requires coprimality. For $(a,b,c)=(2,4,6)$ the
latter gives $16$, whereas dilating $\G_{\{1,2,3\}}=(0123)^\infty$ gives
least period $8$. The factor $\gcd(a,N)$ in the unified formula distinguishes
these values. Quotients with $a/g=1$ are outside the theorem's
hypotheses; use Table~\ref{tab:a1-known} and Lemma~\ref{lem:gcd} instead.
\end{remark}

\begin{remark}\label{rem:pq}
Flammenkamp's Satz~5 \cite{Flammenkamp} gives a floor-function
parametrization agreeing, after translation of his parameter $r$, with
the two parity specializations above. The unified formula uses the base
count $N$ and supplies the reconstruction data of \S\ref{sec:lfunction}.
At the outcome level, \cite[proof of Theorem~5.1, pp.~25--27]{MiklosPost}
derives the additive period word and least period for all positive integers
$a<b$, without a coprimality assumption.\footnote{For the quadratic case,
we use the word in their equation~(21) and the period at the end of the proof;
the quadratic branches
of equations~(19)--(20) contain typographical errors.}
The period is also obtained independently in \cite[Corollary 4]{BLMY}. The period was
claimed without proof in \cite[vol.~3, p.~531]{berlekamp2004winning} and
restated in Flammenkamp's Satz~5; Ho \cite{Ho} records the even-$\eta$
nim-value word from the same source. Moriwaki's thesis
\cite{MoriwakiThesis} gives the periods for both parities and the complete nim-value word
for odd $\eta$. Here the transfer from outcomes to the nim-sequence is
Appendix~\ref{app:additive} and Lemma~\ref{lem:sameperiod}. For $\eta=1$
the generator is $\alpha_n=n+a\lfloor n/a\rfloor+b\lfloor n/\delta\rfloor$,
as in \cite{berlekamp2004winning,Larsson}.
\end{remark}


\section{Reconstruction of the higher nim-values: the \texorpdfstring{$L$}{L}-function}\label{sec:lfunction}

Throughout this section we assume $a\ge2$ and $\gcd(a,b,c)=1$.
We call $S$ \emph{covered} if it is additive
($c=a+b$, Theorem~\ref{thm:additive}) or has an admissible angle
(Theorem~\ref{thm:main}). For these games, the preceding sections
give the $\P$-position set $\W_0=\{x:\G_S(x)=0\}$ explicitly.
We now recover the remaining classes $\W_i=\{x:\G_S(x)=i\}$:
Ferguson's pairing determines $\W_1$, and the $L$-function
distinguishes values $2$ and $3$ on the remaining positions.
We then show that this reconstruction transfers periodicity from
$\W_0$ to the full nim-sequence, and that their least periods agree.
Under Conjecture~\ref{conj:necessity}, the covered games are exactly
the purely periodic games under the standing hypotheses.

By Lemma~\ref{lem:ferguson} and $\G_S\in\{0,1,2,3\}$, the set
\[
  V:=\N\setminus(\W_0\cup\W_1)=\{x:\G_S(x)\in\{2,3\}\}
\]
is exactly the value-$2$ and value-$3$ positions; in particular
$\N=\W_0\sqcup\W_1\sqcup V$ with no further hypothesis, so this partition is exhaustive.

\needspace{15\baselineskip}
\begin{defi}\label{def:L}
Adopt the convention that $y\notin V$ for $y<0$. For $v\in V$ let
\[
  L(v):=\min\{t\in\N:\ v-at\notin V\}-1,
\]
so $L(v)$ counts the successive backward steps of size $a$ that remain
in $V$. The minimum exists, since $v-at<0$ for large $t$. Define the
\emph{candidate} classes
\[
  V_2:=\{v\in V:L(v)\text{ even}\},\qquad V_3:=\{v\in V:L(v)\text{ odd}\},
\]
so that $V=V_2\sqcup V_3$.
\end{defi}

Theorem~\ref{thm:lfunction} below shows $V_2=\W_2$ and $V_3=\W_3$: the $L$-parity
candidates are the true value classes.

Equivalently, in terms of set operations: shift $V$ upward by $a$ repeatedly and count
consecutive collisions. A position of $V$ that survives
$V\cap(V+a)\cap\dots\cap(V+ta)$ but not the next intersection has $L=t$; the value is
$2$ after an even number of collisions and $3$ after an odd number. All the structure
of the nim-values beyond Ferguson's pairing is in this single parity.

Consider a run start $v\in V$, i.e.\ $L(v)=0$, so $v-a\notin V$. A position $<a$ has no move
and lies in $\W_0$, hence $v\ge a$ and $v-a\in\W_0\cup\W_1$. If $v-a\in\W_0$, then Ferguson
(Lemma~\ref{lem:ferguson}) gives $\G_S(v)=1$, contradicting $v\in V$; thus $v-a\in\W_1$, so
$v-2a\in\W_0$ and $\G_S(v-a)=1$. Hence
$\G_S(v)=\operatorname{mex}\bigl(\{1\}\cup\{\G_S(v-s):s\in\{b,c\},\ s\le v\}\bigr)$,
and $\G_S(v)=3$ holds iff $v\ge c$ and $\{\G_S(v-b),\G_S(v-c)\}=\{0,2\}$. The first step is thus to rule out this pair, using the closed form of $\W_0$.

The same descent gives $L$ itself a closed form. For $v\in V$ let $\Pi_a(v)$ be the
largest element of $\W_0$ that is smaller than $v$ and congruent to $v$ modulo $a$.
Descending the $a$-chain from $v$, the first position outside $V$ lies in $\W_1$ by
the argument just given, and the position $a$ below it lies in $\W_0$; the positions
of the class passed on the way lie in $V$ or $\W_1$, so no element of $\W_0$
intervenes. Hence $v\ge2a$ and
\[
  L(v)=\frac{v-\Pi_a(v)}{a}-2 .
\]
Thus $L(v)$ has the same parity as the number of $a$-steps from
$\Pi_a(v)$ to $v$. The run containing $v$ starts at $\Pi_a(v)+2a$; for a
covered game, Theorem~\ref{thm:lfunction} below shows that this run start has
value $2$.

For the non-additive part of the proof, we use the following hypothesis:
\begin{itemize}
\item[\textup{(H)}] either $\W_0(S)=\W_0(B)$ \textup{(}put $p:=a+b$\textup{)}, or
there is $P\in\{c+a,c+b\}$ with $\rho$ admissible for $P$ and $\W_0(S)=T_P$
\textup{(}put $p:=P$\textup{)}.
\end{itemize}
Every covered non-additive $S$ satisfies \textup{(H)} by
Theorem~\ref{thm:verify}.

Under \textup{(H)} the quantity $\Pi_a(v)$ becomes explicit. Write $u:=v\bmod p$;
note $u\ge a$ for $v\in V$, since $u<a$ would give $v\in\W_0$ directly from
\textup{(H)}. By Proposition~\ref{prop:base} in the inert case and
Proposition~\ref{prop:sieveform} in the sieve cases, the fundamental block
$\W_0\cap[0,p)$ is explicit, and it contains $[0,a)$ because positions below $a$
have no move; so the interval $\bigl[p\lfloor v/p\rfloor,\,p\lfloor v/p\rfloor+p\bigr)$
containing $v$ contains a $\P$-position below $v$ congruent to $v$ modulo $a$. Since no element of $\W_0$ in the class of $v$ intervenes
between $\Pi_a(v)$ and $v$, the descent never leaves this interval:
$a\,(L(v)+2)\le u$, hence $\Pi_a(v)=p\lfloor v/p\rfloor+h_a(u)$ and
\[
  L(v)=\frac{u-h_a(u)}{a}-2,\qquad
  h_a(u):=\max\bigl\{h\in\W_0\cap[0,p):\ h<u,\ h\equiv u\ (\mathrm{mod}\ a)\bigr\}.
\]
The maximum ranges over an explicit finite set, so under \textup{(H)} no game
recursion remains in $L$.

\begin{lemma}\label{lem:Vstructure}
Let $S$ be non-additive and assume \textup{(H)}. Write $\W_0:=\W_0(S)$. Then:
\begin{enumerate}
\item[(0)] \textup{(}periodicity\textup{)} $\W_0$, $\W_1$ and $V$ are $p$-periodic from $0$, $L(y+p)=L(y)$ for every
$y\in V$, and no two classes of $\W_0$ differ by $a$, $b$ or $c$ modulo $p$;
\item[(i)] \textup{(}interior\textup{)} $V\cap[0,c)$ is the set of value-$2$ positions of $B$ below $c$; every $y\in V\cap[0,c)$
has $y\bmod(a+b)\in U$, $y\ge b$, $y-2a\in\W_0$, and $L(y)=0$;
\item[(ii)] \textup{(}first seam block\textup{)} in the sieve case, every $y\in V\cap[c,c+a)$ has $L(y)\le1$, and if $L(y)=0$
then the class of $y-2a$ lies in $\W_0$;
\item[(iii)] \textup{(}terminal seam block\textup{)} if $p=c+b$, then no $y\in V\cap[c+a,c+b)$ has $y-b\in V$.
\end{enumerate}
\end{lemma}

\begin{proof}
(0) By \textup{(H)}, $\W_0$ is $p$-periodic from $0$: in the inert case
$\W_0=\W_0(B)$ is $(a{+}b)$-periodic by Proposition~\ref{prop:base}, and in the
sieve case $\W_0=T_P$ is $P$-periodic by construction. For
$\W_1=\W_0+a$: if $x\ge a$ then $x\in\W_1\iff x-a\in\W_0\iff x-a+p\in\W_0\iff
x+p\in\W_1$, and if $x<a$ then $x\in[0,a)\subseteq\W_0$ excludes $x\in\W_1$, while
$x+p\in\W_0$ excludes $x+p\in\W_1$. Hence $\W_1$ and $V$ are $p$-periodic from $0$
as well. For $L$: descending from $y\in V$ in steps of $a$ one reaches
$y-a\lfloor y/a\rfloor\in[0,a)\subseteq\W_0$, so every chain terminates at a non-negative
point, and the $p$-periodicity of $V$ gives $L(y+p)=L(y)$. For the last assertion: two
classes of $\W_0$ differing by $m\in\{a,b,c\}$ modulo $p$ have, after translating one of
them by a multiple of $p$, representatives $x,y\in\W_0$ with $x-y=m$ exactly; then the
$\P$-position $x$ has the $\P$-position option $x-m=y$, contradicting
Lemma~\ref{lem:unique}.

(i) Below $c$ the move $c$ is unavailable, so $\W_0\cap[0,c)=\W_0(B)\cap[0,c)$
\textup{(}by Definition~\ref{def:sieve} in the sieve case, trivially in the inert
case\textup{)} and $\W_1\cap[0,c)=\W_1(B)\cap[0,c)$; hence $V\cap[0,c)$ is the value-$2$
set of $B$ below $c$. Lemma~\ref{lem:descent} gives $y-2a\in\W_0(B)$, and $y-2a<c$ puts it in
$\W_0$. For the residues: $y\notin\W_0(B)\cup\W_1(B)$ means $y\bmod(a+b)\in U$
\textup{(}Remark~\ref{rem:Zfacts}\,(ii)\textup{)}, and $\min U=(\eta+2)a$ for $\eta$ even,
$(\eta+1)a+\varepsilon$ for $\eta$ odd, both at least $b$; hence
$y\ge y\bmod(a+b)\ge b$. Finally $L(y)=0$: if $y-a\in V$ then $y-a<c$ as well, so both
residues lie in $U$; but then $y\bmod(a+b)\in(U+a)\bmod(a+b)\subseteq\mathcal Z$
\textup{(}Remark~\ref{rem:Zfacts}\,(ii)\textup{)}, contradicting $y\bmod(a+b)\in U$.

(ii) If $L(y)\ge2$ then $y-a,y-2a\in V\cap[0,c)$ differ by $a$, which contradicts (i)
exactly as above. If $L(y)=0$ then $y-a\notin V$; $y-a\in\W_0$ would give $y\in\W_1$,
excluded by $y\in V$; so $y-a\in\W_1$, i.e.\ $y-2a\in\W_0$.

(iii) Let $y\in V\cap[c+a,c+b)$ and suppose $y-b\in V$. Since $y-b\in[c-\delta,c)$, part
(i) gives $(y-b)\bmod(a+b)\in U$; as $b\equiv-a\pmod{a+b}$, this reads $(r+a)\bmod(a+b)\in U$ with
$r=y\bmod(a+b)$. Now $U=\Lambda+a$: for $\eta$ even
$\Lambda+a=[(\eta+2)a,(\eta+2)a+\varepsilon)=U$, and for $\eta$ odd
$\Lambda+a=[(\eta+1)a+\varepsilon,(\eta+2)a)=U$. Hence $r\in\Lambda\subseteq\mathcal Z+a$,
so $r\notin\mathcal Z$ and $r-a\in\mathcal Z$. Put $u:=y-c\in[a,b)$. If
$u\notin\mathcal Z$: then $y\notin\W_0(B)$ and $y-c\notin\W_0(B)$, so \textup{(K2)}, which holds
because $\rho\in\Th_{c+b}$ by \textup{(H)} via Lemma~\ref{lem:T2}, gives
$y-b\in\W_0(B)$, i.e.\
$(r+a)\bmod(a+b)\in\mathcal Z$, contradicting $(r+a)\bmod(a+b)\in U$. If $u\in\mathcal Z$: then
$u\ge2a$, since $\mathcal Z\cap[a,2a)=\varnothing$ by the block description of
Proposition~\ref{prop:base}; hence $y\ge c+2a$ and $y-a\in[c+a,c+b)$. Since $y\in V$ we
have $y\notin\W_1$, so the class of $y-a$ is not in $\W_0$; on $[c+a,c+b)$ this reads:
not \textup{(}$(r-a)\in\mathcal Z$ and $(u-a)\notin\mathcal Z$\textup{)}. As
$r-a\in\mathcal Z$, this forces $u-a\in\mathcal Z$; but then $u-a,u\in\mathcal Z$ differ
by $a$, contradicting Remark~\ref{rem:Zfacts}\,(i).
\end{proof}

\begin{lemma}[Run-start exclusion]\label{lem:runstart-na}
Let $S$ be non-additive and assume \textup{(H)}
\textup{(}stated before Lemma~\ref{lem:Vstructure}\textup{)}. Then
for every run start $v$, neither
the class of $v-b$ nor the class of $v-c$ modulo $p$ is an $L$-even element of $V$.
\end{lemma}

\begin{proof}
All memberships are read on classes; the sets and $L$ are $p$-periodic by
Lemma~\ref{lem:Vstructure}\,(0). As shown before Lemma~\ref{lem:Vstructure} we have
$v-2a\in\W_0$. Let $y_0\in[0,p)$ represent the class of $v$; then $y_0\in V$ and the class
of $y_0-2a$ lies in $\W_0$.

We use repeatedly the following consequence of Lemma~\ref{lem:Vstructure}\,(i) and (0):
\begin{itemize}
\item[$(\ast)$] if the class of $v-s$ \textup{(}$s\in\{b,c\}$\textup{)} has a
representative $w\in V\cap[0,c)$, then $w-2a\in\W_0$ and the classes $v-2a$ and
$w-2a$ differ by $s$ modulo $p$, contradicting (0). So no such representative exists.
\end{itemize}

\emph{The class of $v-b$.} If $y_0<c$ then $y_0\ge b$ by
Lemma~\ref{lem:Vstructure}\,(i), and $w:=y_0-b\in[0,c)$ represents the class; if
$y_0\in[c,p)$ then $w:=y_0-b\in[c-b,\,p-b)\subseteq[0,c)$, since $p-b\le c$ in all three
cases of \textup{(H)}. Either way the class of $v-b$ has a representative in $[0,c)$, and
$(\ast)$ shows it is not in $V$; in particular it is not $L$-even in $V$.

\emph{The class of $v-c$.} If $y_0\in[c,p)$ then $w:=y_0-c\in[0,p-c)\subseteq[0,c)$
represents the class, and $(\ast)$ applies. In the inert case every class has a
representative in $[0,p)\subseteq[0,c)$ \textup{(}an inert angle has
$\rho\le b<c$, so $c<a+b$ is impossible and $c=a+b$ is excluded as additive; hence $p=a+b<c$\textup{)}, so this settles it. There remains the sieve case
with $y_0<c$, where the representative is $w:=y_0+(p-c)\in[p-c,\,c+p-c)$, i.e.\
$w=y_0+a<c+a$ for $p=c+a$ and $w=y_0+b<c+b$ for $p=c+b$. If $w<c$, apply $(\ast)$.
If $w\in[c,c+a)$ and $w\in V$: by Lemma~\ref{lem:Vstructure}\,(ii), $L(w)\le1$; if
$L(w)=0$ then the class of $w-2a$ lies in $\W_0$ and differs from that of $v-2a$ by
$p-c\equiv-c$, contradicting (0); so $L(w)=1$, and $w$ is $L$-odd. If
$w\in[c+a,c+b)$ \textup{(}possible only for $p=c+b$\textup{)} and $w\in V$: then
Lemma~\ref{lem:Vstructure}\,(iii) gives $w-b\notin V$; but $w-b=y_0\in V$, a
contradiction, so $w\notin V$. In every case the class of $v-c$ is not an $L$-even
element of $V$.
\end{proof}

\begin{thm}\label{thm:lfunction}
For every covered $S$ \textup{(}additive or meeting the criterion of
Theorem~\ref{thm:main}\textup{)} the value classes $\W_i=\{x:\G_S(x)=i\}$ satisfy
\[
  \W_1=\W_0+a,\quad \W_2=V_2,\quad \W_3=V_3,
\]
with $V_2,V_3$ the $L$-parity candidates of Definition~\ref{def:L}. Equivalently, $\W_0$
together with the $L$-function determines all of $\G_S$.
In the non-additive case, the proof uses the $\P$-set description and
admissibility conditions recorded in \textup{(H)} \textup{(}stated before
Lemma~\ref{lem:Vstructure}\textup{)}.
\end{thm}
\begin{proof}
It suffices to show that every run starts with value $2$.
Consecutive positions in an $a$-chain in $V$ have values in
$\{2,3\}$; since they are joined by a legal move, their values
must alternate. For non-additive games, we use
Lemma~\ref{lem:runstart-na} in a strong induction on the position.
For additive games, the value formula from Appendix~\ref{app:additive}
determines the run-start values.
Write $F$ for the candidate function: $F=0$ on $\W_0$, $F=1$ on $\W_1$,
and $F=2$ resp.\ $3$ on the $L$-even resp.\ $L$-odd part of $V$; the claim is
$\G_S=F$.

\emph{Non-additive $S$ under \textup{(H)}.} We argue by strong induction on $x$, assuming
$\G_S(y)=F(y)$ for all $y<x$. Throughout, $\G_S(x)=0\iff x\in\W_0$ \textup{(}the
definition of the $\P$-positions\textup{)} and $\G_S(x)=1\iff x\in\W_1$
\textup{(}Lemma~\ref{lem:ferguson}\textup{)}; in particular $\G_S(x)\in\{2,3\}$ for
$x\in V$. If $x\in\W_0\cup\W_1$ we are done. Let $x\in V$.

If $L(x)\ge1$: the option $x-a$ lies in $V$, and by induction
$\G_S(x-a)=F(x-a)\in\{2,3\}$; the mex rule gives $\G_S(x)\neq\G_S(x-a)$, so $\G_S(x)$ is
the other element of $\{2,3\}$, which is $F(x)$, since $F$ alternates along the chain by
construction.

If $L(x)=0$ \textup{(}a run start\textup{)} and $x<c$: the move $c$ is unavailable below
$c$, so $\G_S=\G_B$ on $[0,c)$, and $x$ is a value-$2$ position of $B$ by
Lemma~\ref{lem:Vstructure}\,(i); hence $\G_S(x)=2=F(x)$.

If $L(x)=0$ and $x\ge c$: all three moves are available. As shown before
Lemma~\ref{lem:Vstructure}, $x-a\in\W_1$, so $\G_S(x-a)=1$, and
$\G_S(x)=3$ would require a value-$2$ option, necessarily $x-b$ or $x-c$. By induction
$\G_S(x-b)=F(x-b)$ and $\G_S(x-c)=F(x-c)$, and by Lemma~\ref{lem:runstart-na} neither is
$2$ \textup{(}$F$ takes the value $2$ exactly on the $L$-even part of $V$, a
$p$-periodic set\textup{)}. Hence $\G_S(x)=2=F(x)$.

\emph{Additive $S=(a,b,a+b)$.} Here the value function is identified in closed form in
Appendix~\ref{app:additive}: with shifts $(o_0,o_1,o_2,o_3)=(0,\,a,\,-b,\,-\delta)$ one
has, modulo $p$,
\[
  \G_S(x)=F(x):=\min\{\,i:\ x\in\W_0+o_i\ (\mathrm{mod}\ p)\,\},
\]
so $\W_2=\W_0-b$ and $\W_3=(\W_0-\delta)\setminus\W_0$. These coincide with the
$L$-parity candidates $V_2,V_3$. Indeed, a run start $x$ \textup{(}$x\in V$,
$x-a\notin V$\textup{)} cannot lie in $\W_3$: $x\in\W_0-\delta$ gives
$(x-a)+b=x+\delta\in\W_0$, that is $x-a\in\W_0-b=\W_2\subseteq V$. So every run
start has value $2$. Along an $a$-chain in $V$, consecutive positions have values in
$\{2,3\}$ and the mex rule forbids $\G_S(x)=\G_S(x-a)$, so the values alternate;
hence the value is $2$ at even $L$ and $3$ at odd $L$.
\end{proof}

\needspace{12\baselineskip}
\begin{corollary}\label{cor:grundy-period}
Let $S$ be covered, with $p=p_{\mathrm{add}}$ in the additive case
\textup{(}Theorem~\ref{thm:additive}\textup{)} and $p$ the least admissible candidate otherwise \textup{(}Theorem~\ref{thm:main}\textup{)}, so that $x\in\W_0\iff x+p\in\W_0$
for all $x\ge0$. Then $p$ is also a period of the full nim-value sequence from its first term:
\[
  \G_S(x+p)=\G_S(x)\qquad(x\ge0).
\]
In particular the pre-period of $\G_S$ is $q=0$, and the least period of $\G_S$ divides $p$.
\end{corollary}
\begin{proof}
By Theorem~\ref{thm:lfunction}, $\G_S$ is determined pointwise by $\W_0$: value $0$ on
$\W_0$, value $1$ on $\W_1=\W_0+a$, and on $V=\N\setminus(\W_0\cup\W_1)$ the value $2$ or
$3$ according to the parity of $L$. It therefore suffices to show that $\W_0$, $\W_1$, and
the function $L$ on $V$ are invariant under $x\mapsto x+p$.

For $\W_0$ this is the hypothesis. For $\W_1$: if $x\ge a$ then
$x\in\W_1\iff x-a\in\W_0\iff x-a+p\in\W_0\iff x+p\in\W_1$; if $x<a$ then
$x\in[0,a)\subseteq\W_0$, so $x\notin\W_1$ and $x+p\in\W_0$ excludes $x+p\in\W_1$
\textup{(}$\W_0\cap\W_1=\varnothing$ by Lemma~\ref{lem:ferguson}\textup{)}. Hence $V$ is
$p$-periodic from $0$ as well.

For $L$: let $v\in V$ and let $t_0=L(v)+1$ be the least $t$ with $v-at\notin V$. Every
backward $a$-chain terminates at a non-negative point: descending from $v$ in steps of $a$
one reaches the point $v-a\lfloor v/a\rfloor\in[0,a)\subseteq\W_0$, which is not in $V$;
hence $v-at_0\ge0$. By the $p$-periodicity of $V$ on $\N$, for every $t\le t_0$ we have
$v+p-at\in V\iff v-at\in V$, so the chain from $v+p$ terminates at the same depth:
$L(v+p)=L(v)$. Consequently $\G_S(x+p)=\G_S(x)$ for every $x\ge0$, i.e.\ $q=0$; and the
least period of a purely periodic sequence divides every period.
\end{proof}

\begin{corollary}[Seam localization]\label{cor:seam}
Let $S$ be non-additive and covered, with least period $p$. Every position $x$ with
$x\bmod p<c$ satisfies $\G_S(x)=\G_B(x\bmod p)\le2$; consequently the value $3$ occurs
only at positions whose residue lies in $[c,p)$, the seam \textup{(}and not at all in the
inert case, where $p<c$\textup{)}.
\end{corollary}
\begin{proof}
Let $y=x\bmod p<c$. By Corollary~\ref{cor:grundy-period}, $\G_S(x)=\G_S(y)$, and below
$c$ the move $c$ is unavailable, so $\G_S(y)=\G_B(y)\le2$.
\end{proof}

\begin{remark}\label{rem:deferred}
Theorem~\ref{thm:main} uses Corollary~\ref{cor:grundy-period},
Lemma~\ref{lem:sameperiod} and Proposition~\ref{prop:minimality}.
These arguments do not assume the nim-value conclusion of
Theorem~\ref{thm:main}. In the non-additive
case, the $\W_0$-description in \S\ref{sec:rotation}, via \textup{(H)} and
Theorem~\ref{thm:lfunction}, gives pure periodicity in
Corollary~\ref{cor:grundy-period}. Lemma~\ref{lem:sameperiod} then identifies
the least periods, and Proposition~\ref{prop:minimality} proves minimality.
For Theorem~\ref{thm:additive}, the $\W_0$-formula and collision counts are
established first; Appendix~\ref{app:additive} reconstructs the nim-values
from them. Corollary~\ref{cor:grundy-period} then transfers periodicity to
the nim-sequence, whose least period equals that of its zero set.
\end{remark}

\begin{corollary}\label{cor:add-value}
For additive $S=(a,b,a+b)$ with $\gcd(a,b)=1$, period $p$ and $\delta=b-a$, one has modulo $p$
\[
  \W_1=\W_0+a,\qquad \W_2=\W_0-b,\qquad \W_3=(\W_0-\delta)\setminus\W_0,
\]
and equivalently $\G_S(x)=\min\{i:x\in\W_0+o_i\}$ with
$(o_0,o_1,o_2,o_3)=(0,a,-b,-\delta)$, membership again read modulo $p$.
\end{corollary}

Explicit closed forms for $\W_2$ and $\W_3$ follow by substituting the closed form of
$\W_0$ from Theorem~\ref{thm:additive}.
\begin{proof}
This is the additive case established in the proof of Theorem~\ref{thm:lfunction}, where the
value-$2$ and value-$3$ classes were identified as $R_2=\W_0-b$ and
$R_3=(\W_0-\delta)\setminus\W_0$.
\end{proof}

\begin{lemma}\label{lem:sameperiod}
Let $S$ be non-additive and satisfy \textup{(H)} \textup{(}stated before
Lemma~\ref{lem:Vstructure}\textup{)}, or additive. Then the least period of $\G_S$
equals the least period of $\W_0(S)$.
\end{lemma}
\begin{proof}
The reconstruction of Theorem~\ref{thm:lfunction} expresses $\G_S$ as a function of $\W_0$
and the least move alone, and that function commutes with translation of $\N$; so a period
of $\W_0$ is a period of $\G_S$, while conversely $\W_0$ is a level set of $\G_S$. We spell
this out, using twice the fact that for a sequence purely periodic from $0$ the least
period divides every period: if $s<t$ are periods, write $\gcd(s,t)=mt-ns$ with $m,n\ge0$;
then $f(x+\gcd(s,t))=f(x+\gcd(s,t)+ns)=f(x+mt)=f(x)$, every shift being forward, so
$\gcd(s,t)$ is a period, and minimality forces divisibility. Let $p_0$ be the least period
of $\W_0$ and $p^{*}$ that of $\G_S$. The proof of
Corollary~\ref{cor:grundy-period} uses only that $\W_0$ is periodic from $0$ with the period
in question, so it applies with $p_0$: the nim-sequence is $p_0$-periodic from $0$, whence
$p^{*}\mid p_0$. Conversely $\W_0$ is a level set of $\G_S$, so every period of $\G_S$ is a
period of $\W_0$ and $p_0\mid p^{*}$. Hence $p^{*}=p_0$. Neither step uses
Proposition~\ref{prop:minimality}.
\end{proof}

\subsection*{Minimality of the period}

We now show that the periods returned by the theorems are least. By
Lemma~\ref{lem:sameperiod} this is a statement about $\W_0$ alone. The mechanism is a
quotient construction: a proper period would force the explicit cyclic $\W_0$-pattern
to be a power of a shorter word, which would itself have to solve the three-move
recurrence, and the known pattern forbids this. The gap-window set $\Omega$, the
modulus $\mu$ and the two lemmas quoted below are stated and proved in
Appendix~\ref{app:intervals}.

\begin{prop}[Minimality]\label{prop:minimality}
Let $S$ be covered. Put $p=p_{\mathrm{add}}$ in the additive case; in the
non-additive case let $P$ be the least admissible candidate and put $p=P$. Then $p$ is
the least period of $\W_0(S)$ and of $\G_S$: the nim-value
sequence and the outcome sequence have the same least period $p$.
\end{prop}
\begin{proof}
\emph{Additive case.} The least period of $\W_0$ is $p_{\mathrm{add}}$ by
the gap argument in the proof of Theorem~\ref{thm:additive}.

\emph{Inert case $\rho\in\Th_{a+b}$, $p=a+b$.} Here $\W_0(S)=\W_0(B)$, whose least period
by \cite[Theorem 3.4]{MiklosPost} is $a+b$ unless $b$ is an odd multiple of $a$; the
harmonic exception is excluded exactly as in the proof of Theorem~\ref{thm:main}
\textup{(}there $\mathcal I\subseteq a\Z$, so $a\mid c$\textup{)}.

\emph{Phase case $\rho\notin\Th_{a+b}$, $p\in\{c+a,c+b\}$.} Suppose the least period
$p_0$ of $\W_0(S)$ satisfies $p_0<p$, and write $u=v^{M}$, $M\ge2$, as in
Lemma~\ref{lem:quotient}(a). The cyclic word $v$ solves the three-move recurrence modulo
$p_0$ \textup{(}Lemma~\ref{lem:quotient}(b)\textup{)}; in particular its support
$Y=\{x\in[0,p_0):v(x)=1\}$ satisfies $(Y+a)\cap Y=(Y+b)\cap Y=\emptyset$ in $\Z/p_0\Z$:
if $y$ and $y+a$ were both in $Y$, the recurrence at $y+a$ would force $v(y)=0$. By
Lemma~\ref{lem:quotient}(c), $Y=\{x\in[0,p_0):x\bmod(a+b)\in\mathcal Z\}$, so
Lemma~\ref{lem:prefix} applies and yields $\mu\mid p_0$; with $p_0\mid p$ this gives
$\mu\mid p$.

If $\mu=a+b$, then $c+a$ is a multiple of $a+b$ only for $\rho=b$, and $c+b$ only for $\rho=a$; both angles
lie in $\mathcal I=\Th_{a+b}$ \textup{(}Remark~\ref{rem:Zfacts}\textup{)}, contradicting
$\rho\notin\Th_{a+b}$. If $\mu=2a$ \textup{(}$b=ka$, $k$ odd\textup{)}: $2a\mid c+a$ and
$2a\mid c+ka$ each give $c\equiv a\pmod{2a}$, hence $a\mid c$, contradicting
$\gcd(a,b,c)=1$ with $a\ge2$. In every case $p_0=p$, and Lemma~\ref{lem:sameperiod}
transfers this to $\G_S$.
\end{proof}

\begin{corollary}\label{cor:zones-disjoint}
Let $S$ be a primitive non-additive ruleset with $a\ge2$, and put $\rho=c\bmod(a+b)$.
If $\rho\notin\Th_{a+b}$, then $\rho$ belongs to at most one of $\Th_{c+a}$ and
$\Th_{c+b}$. Hence for such $S$ a non-inert admissible angle is admissible for
exactly one of the periods $c+a$, $c+b$,
and an admissible angle determines its least period unambiguously.
\end{corollary}

\begin{proof}
Suppose $\rho\notin\Th_{a+b}$ is admissible for both $c+a$ and $c+b$. The least
admissible candidate is $c+a$, so Proposition~\ref{prop:minimality} makes $c+a$ the
least period of $\W_0(S)$; but Theorem~\ref{thm:verify} applied with $P=c+b$ shows
that $c+b$ is also a period, so $c+a$ divides $c+b$ and hence divides
$(c+b)-(c+a)=\delta$, contradicting $0<\delta<c+a$.
\end{proof}

\begin{remark}\label{rem:no-collapse}
The case $\mu=2a$ of the proof says, concretely, that on the harmonic rows $b=ka$ the
least period never falls to a proper divisor of $c+b$: such a divisor would force
$a\mid c$, which $\gcd(a,b,c)=1$ with $a\ge2$ forbids. For $a=1$ nothing is forbidden
and the collapse is real; see Remark~\ref{rem:a1-fails}.
\end{remark}

\begin{remark}
For general subtraction games the nim-value period can be an unbounded multiple of the
outcome period (Flammenkamp's Satz~22, \cite{Flammenkamp});
Proposition~\ref{prop:minimality} shows that for the covered three-move games this
ratio collapses to~$1$; that it extends to the whole purely periodic regime would
follow from Conjecture~\ref{conj:necessity}.
\end{remark}

\section{Necessity: the conjecture in rotation form}\label{sec:necessity}

Theorem~\ref{thm:main} gives pure periodicity and the least period whenever
an angle is admissible. We conjecture the converse. Recall that the
candidate periods are ordered as $a+b<c+a<c+b$.

\begin{conj}\label{conj:necessity}
Let $a\ge2$, $\gcd(a,b,c)=1$ and $c\neq a+b$. Then $\G_S$ is purely periodic if and only if
the angle is admissible for at least one candidate period, that is, if and only if
\[
  \rho\ \in\ \Th_{a+b}\cup\Th_{c+a}\cup\Th_{c+b} ,
\]
and in that case the least period is the least admissible one,
\[
  p=\min\big\{P\in\{a+b,\,c+a,\,c+b\}:\ \rho\in\Th_P\big\}.
\]
\end{conj}

Only necessity remains open; sufficiency and the least-period
formula follow from Theorem~\ref{thm:main}.
As a finite check, we examined all primitive non-additive triples
with $2\le a\le8$, $a<b\le40$, and $b<c\le100$.
None of the $7{,}828$ non-admissible cases was purely periodic.
All $7{,}557$ admissible cases were purely periodic, with least
nim-value period equal to the least admissible candidate.
The accompanying script detects the first repeated complete
nim-value state of length $c$; every case reached such a state.

In terms of Table~\ref{tab:closedforms}, the conjecture predicts
positive pre-period outside the non-additive rows. For $b\le2a$ the uncovered angles are
$0<\rho<\delta$; for $b>2a$ and $a\nmid b$ they are the complement of
$\{a,b,b-2a\}\cup C_{a,b}$. If $b=ka$ with $k\ge3$ odd, every primitive
angle is covered, whereas for even $k\ge4$ none is covered.
On harmonic shapes, primitivity requires $\gcd(a,\rho)=1$.
Below we separate the proved cases from two conjectures which together
imply Conjecture~\ref{conj:necessity} for $c\ge2(a+b)$.

\begin{remark}[the hypotheses are essential]\label{rem:gcd-essential}
The hypotheses $a\ge2$ and $\gcd(a,b,c)=1$ matter for the least-period
formula. The first is illustrated in Remark~\ref{rem:a1-fails}. For the
second, take $S=\{2,14,c\}$ with $c\equiv2\pmod4$. Then
$\rho\in\{2,6,10,14\}=\mathcal I$, so the angle criterion without the
primitive hypothesis would return least period $16$.
But $S=2\{1,7,c/2\}$, whose quotient has only odd moves and hence nim-value
$x\bmod2$. Lemma~\ref{lem:gcd} gives
$\G_S(x)=\lfloor x/2\rfloor\bmod2$, with least period $4$.
Thus both the least-period formula and the equality in
Conjecture~\ref{conj:period} can fail while pure periodicity survives.
Non-primitive games must first be reduced by Lemma~\ref{lem:gcd};
quotients with least move $1$ are handled by the known Grundy
classification in Remark~\ref{rem:a1}, whose purely periodic cases
are listed in Table~\ref{tab:a1-known}.
\end{remark}

\subsection*{The period \texorpdfstring{$a+b$}{a+b}}

\begin{lemma}\label{lem:smallperiod}
Suppose $\G_S$ is purely periodic with least period $p\le a+b$, and $c\ge2(a+b)$. Then
$\W_0(S)=\W_0(B)$ and $\rho\in\Th_{a+b}$.
\end{lemma}

\begin{proof}
Below $c$ the games $S$ and $B$ have the same options, hence the same
$\P$-positions. Their common indicator word $u$ of length $c$ has periods
$p$ and $a+b$. Put $g=\gcd(p,a+b)$. Since
$c\ge2(a+b)\ge p+(a+b)-g$, the theorem of Fine and Wilf
\cite{FineWilf} gives period $g$ for $u$.
As $g\mid p$ and $c\ge p+g$, repeating the first $p$ letters shows that
$\W_0(S)$ is $g$-periodic from $0$, hence $(a+b)$-periodic.
It agrees with the $(a+b)$-periodic base set on $[0,a+b)$, so the two
sets coincide. Proposition~\ref{prop:inert-iff} then gives
$\rho\in\Th_{a+b}$.
\end{proof}

\begin{prop}\label{prop:necessity-inert}
Let $c\ge2(a+b)$. Then $\G_S$ is purely periodic with least period $p\le a+b$ if and only if
$\rho\in\Th_{a+b}$; in that case $\W_0(S)=\W_0(B)$ and $p$ is the least period of $\G_B$, a
divisor of $a+b$.
\end{prop}

\begin{proof}
The forward implication is Lemma~\ref{lem:smallperiod}.
Conversely, $\rho\in\Th_{a+b}$ gives $\G_S=\G_B$ by
Proposition~\ref{prop:inert-iff}; its least period divides $a+b$
by Remark~\ref{rem:base-period}.
\end{proof}

Lemma~\ref{lem:idiff} explicitly lists these angles through
$\Th_{a+b}=\mathcal I$.

\subsection*{The periods \texorpdfstring{$c+a$ and $c+b$}{c+a and c+b}}

\begin{conj}[period restriction]\label{conj:period}
Let $S$ be non-additive with $a\ge2$ and $\gcd(a,b,c)=1$. If $\G_S$ is purely periodic then
its least period satisfies $p\in\{a+b,\,c+a,\,c+b\}$.
\end{conj}

This sharpens Ward's conjecture \cite{Ward} in the purely periodic
regime: Ward predicts that the least nim-value period of any
non-additive three-move game divides a pair sum and equals the gcd of
the pair sums it divides; here it is asserted to equal a pair sum.
At the outcome level, Flammenkamp \cite[Vermutung 6]{Flammenkamp}
conjectured pair-sum divisibility with gcd alternatives in 1997, and
\cite[Conjecture 1]{MiklosPost} predicts a least outcome period below
$2c$. The outcome period divides the nim-value period, so
Conjecture~\ref{conj:period} implies the outcome divisibility assertions
in this regime; for covered games the periods coincide
(Proposition~\ref{prop:minimality}). The general form of these
predictions originates in those works, Alth\"ofer and B\"ultermann
\cite{AB95}, and Flammenkamp's computational classification
\cite{Flammenkamp}. Lemma~\ref{lem:smallperiod} settles $p\le a+b$
when $c\ge2(a+b)$; the remaining cases are $p>a+b$ and the shorter
range $b<c<2(a+b)$.

For $b\le2a$, the proved least-period regions agree after normalization
with the integral specialization of \cite[Theorem 3]{Moriwaki};
that theorem also excludes least period $c+b$ when $b=2a$.
For $b\le2a$ and $0<\rho<\delta$, Moriwaki rules out pure periodicity with any of the
three candidate sums as period and conjectures that no period works.
Conjecture~\ref{conj:period} would therefore settle that region.
On $b=2a$ this is already known: for $0<\rho<a$,
\cite[Proposition 3.1]{Zhang} gives $q=c+a-\rho$ and the eventual
period. Together with Theorem~\ref{thm:main}, this proves
Conjecture~\ref{conj:necessity} on that line.

Further proved cases are the three ultimately bipartite families of
\cite[Theorem 6.3]{Zhang}, for odd $a\ge3$ and $t\ge1$:
\[
 \{a,a+2,(2a+2)t+1\},\qquad
 \{a,2a-1,(3a-1)t+a-2\},\qquad
 \{a,2a+1,(3a+1)t-1\}.
\]
Their eventual least nim-value period is $2$; since
$\G_S(0)=\G_S(1)=0$, they cannot be purely periodic.
They have angles $\rho=1$, $\rho=\delta-1$ and $\rho=a+b-1$,
respectively, and exhaust the corresponding angle classes with $c>b$.
For $b=a+2$ and odd $a$, the only uncovered angle is $\rho=1$,
so Conjecture~\ref{conj:necessity} holds for that entire subfamily.

\begin{conj}[admissible angle]\label{conj:residue}
Let $S$ be non-additive with $a\ge2$ and $\gcd(a,b,c)=1$, and suppose that
$\G_S$ is purely periodic with least period $p$. If $p=c+a$ then
$\rho\in\Th_{c+a}$, and if $p=c+b$ then $\rho\in\Th_{c+b}$.
\end{conj}

\begin{remark}\label{rem:MP-conj3}
Mikl\'os and Post \cite[Conjecture 3]{MiklosPost} predict an arithmetic
characterization of zero outcome pre-period and least outcome period
$b+c$, using successive divisions of $b$ by $a$ and of $c,c-a$ modulo
$a+b$ and then $2a$. Under the present primitive non-additive
hypotheses, Theorem~\ref{thm:main} and Proposition~\ref{prop:minimality}
give that least period on
$\Th_{c+b}\setminus(\Th_{a+b}\cup\Th_{c+a})
=\Th_{c+b}\setminus\Th_{a+b}$
(Corollary~\ref{cor:zones-disjoint}).
Equivalence of this region with their arithmetic conditions would give
the sufficiency half of their prediction, by
Theorems~\ref{thm:main} and~\ref{thm:thetacb}.
That equivalence remains unproved here.
\end{remark}

\begin{prop}\label{prop:reduction}
Conjectures~\ref{conj:period} and~\ref{conj:residue} together imply
Conjecture~\ref{conj:necessity} and the converse of Theorem~\ref{thm:main}, for all
$c\ge2(a+b)$.
\end{prop}

\begin{proof}
Assume the two conjectures, and let $S$ satisfy the hypotheses of
Conjecture~\ref{conj:necessity}, with $c\ge2(a+b)$ and $q=0$.
The period conjecture gives $p\in\{a+b,c+a,c+b\}$.
For $p=a+b$, Proposition~\ref{prop:necessity-inert} gives
$\rho\in\Th_{a+b}$; for $p=c+a$ or $c+b$, the angle conjecture gives
membership in the corresponding admissible set. Thus the criterion
holds. Its converse, including the least-period assertion, is
Theorem~\ref{thm:main}.
\end{proof}

\appendix

\section{Interval computations}\label{app:intervals}

We compute the admissible angles from the base pattern $\mathcal Z$, then
prove the two lemmas used for minimality. Put
\[
  \mu:=\begin{cases}2a,&b=ka\text{ with }k\text{ odd},\\
  a+b,&\text{otherwise}.
  \end{cases}
\]
In the first case the base pattern repeats with step $2a$. Let
\[
  \Omega:=\{x\in[0,a+b):([x,x+a)\bmod(a+b))\cap\mathcal Z=\emptyset\}
\]
be the set of starts of the cyclic length-$a$ windows avoiding $\mathcal Z$.
By Lemma~\ref{lem:avoid}, it is the complement of $\mathcal Z-[0,a)$.

\begin{lemma}[Gap windows]\label{lem:windows}
The window set $\Omega$ is as follows:
\begin{enumerate}
\item[(i)] $\eta=0$ \textup{(}i.e.\ $b<2a$\textup{)}: $\Omega=[a,b]$;
\item[(ii)] $\varepsilon\ge1$, $\eta\ge2$ even:
  $\Omega=\{(2j-1)a:1\le j\le \eta/2\}\cup\big[(\eta+1)a,\,(\eta+1)a+\varepsilon\big]$;
\item[(iii)] $\varepsilon\ge1$, $\eta$ odd:
  $\Omega=\{(2j-1)a:1\le j\le (\eta+1)/2\}\cup\{(\eta+1)a+\varepsilon\}$;
\item[(iv)] $b=ka$, $k$ odd: $\Omega=\{(2j-1)a:1\le j\le (k+1)/2\}$;
\item[(v)] $b=ka$, $k$ even: $\Omega=\{(2j-1)a:1\le j\le k/2-1\}\cup\big[(k-1)a,\,ka\big]$.
\end{enumerate}
In particular $\min\Omega=a$, so $0\notin\Omega$.
\end{lemma}

\begin{proof}
Since $[0,a)\subseteq\mathcal Z$, no complementary run crosses $0$.
A run $[u,u+\ell)$ contains a length-$a$ window precisely when $\ell\ge a$,
and its possible starts are $[u,u+\ell-a]$. Apply this rule to the
complementary blocks in Proposition~\ref{prop:base}: the ordinary odd
blocks have length $a$; the top run has length $b$ in (i),
$a+\varepsilon$ in (ii), $a$ in (iii), and $2a$ in (v).
In (iv) all complementary runs have length $a$. This gives the five formulas.
\end{proof}

\begin{prop}\label{prop:evenharm}
Let $b=ka$ with $k$ even. Then every element of $\Th_{c+b}$ is a multiple of $a$.
Consequently, for $a\ge2$ and $\gcd(a,b,c)=1$, no ruleset with $b=ka$, $k$ even, has an
admissible angle for the period $c+b$.
\end{prop}

\begin{proof}
Here $\Lambda=[(k-1)a,ka)\subseteq D$. The first clause of
\eqref{eq:T2} requires
$(\rho+(k-1)a)\bmod(a+b)\in\Omega$.
By Lemma~\ref{lem:windows}(v), either $a\mid\rho$ or $\rho\in[0,a]$.
If $0<\rho<a$, the point $d=(k-1)a\in D$ has
$d+\rho\in\Lambda$, violating the second clause of \eqref{eq:T2}.
Thus $a\mid\rho$ in every case. Since $a\mid(a+b)$, this implies
$a\mid c$, contrary to $\gcd(a,b,c)=1$ when $a\ge2$.
\end{proof}

\begin{proof}[Proof of Lemma~\ref{lem:idiff} and the remaining cases of Theorem~\ref{thm:arcs}]
For $b\le2a$, one has $\mathcal Z=[0,a)$, so a translate avoids
$\mathcal Z$ exactly at shifts in $[a,b]$; the $c+a$ assertion is
Proposition~\ref{prop:theta-ca-small}. If $b=ka$ with $k$ odd, the
pattern modulo $2a$ is $[0,a)$, and avoidance holds exactly at odd
multiples of $a$. Denote this set by $\mathcal O$; thus
$\mathcal I=\mathcal O$ in the odd harmonic case, independently of
the $c+a$ calculation below.

For $b>2a$, put $A=\mathcal Z\cap[0,\delta)$ and set
$E=\mathcal O$ in the odd harmonic case and $E=\{a,b\}$ otherwise.
We prove the common difference identity
\[
  A-\mathcal Z=\Z/(a+b)\Z\setminus E.
\]
The differences $a,b$ are excluded by Remark~\ref{rem:Zfacts}(i).
In the odd harmonic case every difference reduces modulo $2a$ to
$r-r'$ with $0\le r,r'<a$, so every element of $\mathcal O$ is excluded.

For coverage, let $\sigma\notin E$. If some $\alpha\in[0,a)$ has
$(\alpha-\sigma)\bmod(a+b)\in\mathcal Z$, then $\sigma\in A-\mathcal Z$.
Otherwise $\omega=(a+b-\sigma)\bmod(a+b)$ belongs to $\Omega$.
In the odd harmonic case Lemma~\ref{lem:windows}(iv) would make
$\sigma$ an odd multiple of $a$, a contradiction. In the other cases,
substituting the windows from Lemma~\ref{lem:windows} and omitting
$\sigma=a,b$ leaves exactly the following possibilities. Each row
gives witnesses $\alpha\in A$, $z\in\mathcal Z$ with $\alpha-z=\sigma$:
\[
\begin{array}{c|c|c|c}
\text{case}&\sigma&\alpha&z\\ \hline
\varepsilon\ge1,\ \eta\text{ even}
 &(2i+1)a+\varepsilon,\ 1\le i\le\eta/2-1&(2i+2)a&a-\varepsilon\\
 &a+s,\ 1\le s\le\varepsilon&2a&a-s\\[2pt]
\varepsilon\ge1,\ \eta\text{ odd}
 &2ia+\varepsilon,\ 1\le i\le(\eta-1)/2&2ia+\varepsilon&0\\[2pt]
\varepsilon=0,\ \eta\text{ odd}
 &2ia,\ 2\le i\le(\eta-1)/2&2ia&0\\
 &a+s,\ 1\le s\le a&2a&a-s
\end{array}
\]
Empty index ranges contribute nothing. In the even-$\eta$ rows, a
possible endpoint $\alpha=\eta a$ lies in the partial block of $A$,
$[\eta a,\eta a+\varepsilon)$; all other listed points lie in full
blocks of $\mathcal Z$.
All listed $\alpha$ are smaller than $\delta$, so the identity follows.

Since $A\subseteq\mathcal Z$, this identity gives
$\mathcal I\subseteq E$; the reverse inclusion follows from the
non-collisions already noted. This proves Lemma~\ref{lem:idiff}.
By Proposition~\ref{prop:instances}, the complement of
$\delta-(A-\mathcal Z)$ is $\Th_{c+a}$, hence $\Th_{c+a}=\delta-E$.
For $E=\{a,b\}$ this is $\{\delta-a,b\}$, since
$\delta-b\equiv b\pmod{a+b}$; in the odd harmonic case both $\delta$
and $a+b$ are even multiples of $a$, so $\delta-\mathcal O=\mathcal O=\mathcal I$.
This proves the remaining cases of Theorem~\ref{thm:arcs}.
\end{proof}

\begin{proof}[Proof of Proposition~\ref{prop:arcdiff}]
Write $t=\lfloor\eta/2\rfloor$. For nonempty integer intervals,
Lemma~\ref{lem:avoid} gives
\[
 [u,u+\ell)-[v,v+m)=[u-v-m+1,u-v+\ell-1],
\]
with the resulting closed interval read modulo $a+b$.
The full blocks of $\mathcal Z$ and $D$ are
$Z_j=[2ja,(2j+1)a)$ and $D_k=[(2k-1)a,2ka)$.
For example, when $\eta=2t$, the interval $Z_j-\Lambda$ is
\[
 [(2j-\eta-1)a-\varepsilon+1,(2j-\eta)a-1]
 \equiv[(2j+1)a+1,(2j+2)a+\varepsilon-1].
\]
The other substitutions give the table below. For odd $\eta$, write
$Z_*=[(\eta+1)a,(\eta+1)a+\varepsilon)$ for the partial block.
\[
\begin{array}{c|c|c|c}
\eta&\text{difference}&\text{closed arc}&\text{index}\\ \hline
\text{even}&Z_j-\Lambda&[(2i+1)a+1,(2i+2)a+\varepsilon-1]&i=j\in[0,t]\\
&\Lambda-D_k&[(2i+1)a+1,(2i+2)a+\varepsilon-1]&i=t-k\in[0,t-1]\\
&\Lambda-\Lambda&[1-\varepsilon,\varepsilon-1]&\\[2pt]
\text{odd}&Z_j-\Lambda&[2ia+1,(2i+2)a-\varepsilon-1]&i=j-t-1\in[-t-1,-1]\\
&Z_*-\Lambda&[1,a-1]&\\
&\Lambda-D_k&[(2i+1)a+\varepsilon+1,(2i+3)a-1]&i=t-k\in[-1,t-1]
\end{array}
\]
Here $0\le j\le t$ and $1\le k\le\lceil\eta/2\rceil$;
$D$ has the additional block $\Lambda$ exactly when $\eta$ is even.
Taking the indicated unions proves both formulas, including empty
index ranges. The assumption $a\nmid b$ ensures $1\le\varepsilon<a$,
so every input block used in the difference formula is nonempty.
\end{proof}

\begin{proof}[Proof of Theorem~\ref{thm:thetacb}]
By Proposition~\ref{prop:instances}, $\Th_{c+b}$ is the complement of
$(\mathcal Z-\Lambda)\cup(\Lambda-D)$.
For $\eta=2t$, Proposition~\ref{prop:arcdiff} gives this union as
\[
 \bigcup_{i=0}^{t}[(2i+1)a+1,(2i+2)a+\varepsilon-1]
 \ \cup\ [1-\varepsilon,\varepsilon-1].
\]
The gaps between consecutive arcs are $[2ia+\varepsilon,(2i+1)a]$
for $1\le i\le t$. The top arc ends at $a+b-1$, the first starts
at $a+1$, and the arc through $0$ leaves the remaining gap
$[\varepsilon,a]$. These are exactly $C_{a,b}\sqcup\{a\}$.

For $\eta=2t+1$, translate each indexed arc of $\mathcal Z-\Lambda$ by
$a+b$: the index becomes $i'=i+t+1\in[0,t]$ and the arc has the
same form as those of $\Lambda-D$. The extra arc $[1,a-1]$ is
contained in $[-a+\varepsilon+1,a-1]$, so the union is
\[
 \bigcup_{i=-1}^{t}[(2i+1)a+\varepsilon+1,(2i+3)a-1].
\]
Its consecutive gaps are $[(2j-1)a,(2j-1)a+\varepsilon]$ for
$1\le j\le t+1$. At the join around the circle, the last endpoint
is $(\eta+2)a-1$ and the next start is
$(\eta+1)a+2\varepsilon+1$, leaving
$[(\eta+2)a,(\eta+1)a+2\varepsilon]$.
This gap is nonempty exactly when $a\le2\varepsilon$.
Separating $a$ from the first gap again gives $C_{a,b}\sqcup\{a\}$.

For $b>2a$, all nonterminal gaps lie below $b$ and the terminal
odd-case gap lies above $b$; after removing $a$, they avoid
$\mathcal I=\{a,b\}$. For $b<2a$, one has $\eta=0$ and
$C_{a,b}=[\delta,a)$, disjoint from $\mathcal I=[a,b]$.
This proves $\Th_{c+b}\setminus\Th_{a+b}=C_{a,b}$.
Finally, if $b=ka$ with $k$ odd, then $\Lambda=\emptyset$ and both
clauses of \eqref{eq:T2} are vacuous, giving the harmonic assertion.
\end{proof}

\begin{lemma}\label{lem:quotient}
Let $S$ be covered, with period $p$ as in Theorems~\ref{thm:additive}
and~\ref{thm:main} \textup{(}$p=p_{\mathrm{add}}$ in the additive case, the least admissible candidate otherwise\textup{)}, let
$u\in\{0,1\}^{p}$ be the cyclic indicator word of $\W_0(S)\cap[0,p)$, and let $p_0$ be
the least period of $\W_0(S)$. Then:
\begin{enumerate}
\item[(a)] $p_0\mid p$, and $u$ is invariant under cyclic rotation by $p_0$;
equivalently $u=v^{M}$ with $v=u|_{[0,p_0)}$ and $M=p/p_0$.
\item[(b)] The cyclic word $v$ satisfies the three-move recurrence with indices read
modulo $p_0$: for every $x\in\Z/p_0\Z$,
\[
  v(x)=1\iff v(x-a)=v(x-b)=v(x-c)=0 .
\]
\item[(c)] If $S$ is non-additive and $p_0<p$, then $p_0\le p/2<c$, so on
$[0,p_0)$ the word $v$ coincides with the base pattern: $v(x)=1\iff x\in\W_0(B)$.
\end{enumerate}
\end{lemma}
\begin{proof}
(a) The least period of a sequence periodic from $0$ divides every
period, so $p_0\mid p$ and the length-$p$ word is the indicated power.
For (b), choose a representative in $[c,c+p)$ of each class modulo
$p$. All three moves are legal there, so the actual $\P$-position
recurrence holds cyclically for $u$; substituting
$u(x)=v(x\bmod p_0)$ gives the recurrence for $v$.
For (c), a proper divisor satisfies
$p_0\le p/2\le(c+b)/2<c$. Below $c$ the move $c$ is unavailable,
so $\W_0(S)$ agrees with the base pattern on $[0,p_0)$.
\end{proof}

Thus a proper period would make a prefix of the base pattern satisfy
the cyclic recurrence. The next lemma restricts such prefix lengths.
\begin{lemma}[Prefix rigidity]\label{lem:prefix}
Let $d\ge1$ and let $Y:=\{x\in[0,d):\ x\bmod(a+b)\in\mathcal Z\}$, regarded as a subset of
$\Z/d\Z$. Then
\[
  (Y+a)\cap Y=\emptyset\ \text{ and }\ (Y+b)\cap Y=\emptyset\quad\text{in }\Z/d\Z
  \qquad\iff\qquad \mu\mid d .
\]
\end{lemma}

\begin{proof}
($\Leftarrow$) If $(a+b)\mid d$ and $x,y\in Y$ satisfy $x-y\equiv a\pmod d$, then
$x-y\equiv a\pmod{a+b}$, so the residues of $y,x$ give an element of
$\mathcal Z\cap(\mathcal Z+a)=\emptyset$ \textup{(}Remark~\ref{rem:Zfacts}\,(i)\textup{)},
a contradiction; likewise for $b\equiv-a\pmod{a+b}$. If
$\mu=2a\mid d$ \textup{(}$b=ka$, $k$ odd\textup{)}, then every element of $Y$ is
congruent modulo $2a$ to an element of $[0,a)$, and two such elements cannot differ by
$a\equiv b\pmod{2a}$.

($\Rightarrow$) Assume both disjointnesses. If $d\le a$ then $Y=\Z/d\Z$ and
$(Y+a)\cap Y=\emptyset$ fails; so $d>a$ and $[0,a)\subseteq Y$. Write $d=Q(a+b)+\bar d$ with $0\le\bar d<a+b$. The hypotheses enter only through two consequences, one per move, and each
of them is an avoidance statement on the circle $\Z/(a+b)\Z$ in the sense of
Lemma~\ref{lem:avoid}.

\emph{Claim 1 \textup{(}the last $a$-window is empty\textup{)}.} In $\Z/d\Z$ one has $x-a\equiv x+(d-a)$, so
$(Y+a)\cap Y=\emptyset$ applied to $[0,a)\subseteq Y$ forces $[d-a,d)\cap Y=\emptyset$: the
last $a$ positions carry no element of $\mathcal Z$. If $1\le\bar d\le a$ this already
fails, since $d>a$ gives $Q\ge1$ and then $Q(a+b)=d-\bar d$ lies in $[d-a,d)$ with residue
$0\in\mathcal Z$. Hence $\bar d=0$, or else $\bar d>a$ and the window $[d-a,d)$ has the
literal residues $[\bar d-a,\bar d)$, so that
\begin{equation}\label{eq:R1}
  \bar d-a\in\Omega .
\end{equation}

\emph{Claim 2 \textup{(}if $d>b$ then $(\bar d+a)\bmod(a+b)\in\mathcal I$\textup{)}.} Since $\max\mathcal Z\le b-1$, every residue
$z\in\mathcal Z$ occurs in $Y$ as its least representative $z<b<d$, and $z+(d-b)\in[0,d)$
differs from $z$ by $-b$ in $\Z/d\Z$; hence $z+(d-b)\notin Y$. As $d-b\equiv\bar d+a\pmod{a+b}$, this says $(z+\bar d+a)\bmod(a+b)\notin\mathcal Z$ for \emph{every} $z\in\mathcal Z$, that
is,
\begin{equation}\label{eq:R2}
  (\bar d+a)\bmod(a+b)\in\mathcal I .
\end{equation}

If $\bar d=0$ then $(a+b)\mid d$, and $\mu\mid(a+b)\mid d$ \textup{(}for $b=ka$ with $k$ odd, $a+b=(k+1)a$ with $k+1$ even\textup{)}: done. Assume henceforth $\bar d\neq0$, so that
\eqref{eq:R1} holds.

\emph{The harmonic shape $b=ka$, $k$ odd.} Lemma~\ref{lem:windows}(iv) makes
$\bar d-a$ an odd multiple of $a$, so $2a\mid\bar d$; with $2a$ dividing $a+b$ this gives
$\mu=2a\mid d$, as required. In every remaining case, $b$ is not an odd
multiple of $a$ and $\mu=a+b$. We derive a contradiction from
$\bar d\neq0$, considering $d>b$ and $a<d\le b$ separately.

\emph{Claim 3 \textup{(}the case $d>b$\textup{)}.} If $\varepsilon\ge1$ with $\eta\ge1$, or $b=ka$ with $k$ even,
$k\ge4$, then Lemma~\ref{lem:idiff} turns \eqref{eq:R2} into $(\bar d+a)\bmod(a+b)\in\{a,b\}$, so $\bar d\in\{0,\delta\}$; and $\bar d=\delta$ contradicts \eqref{eq:R1},
because $\delta-a\notin\Omega$: for $\varepsilon\ge1$ the number
$\delta-a=(\eta-1)a+\varepsilon$ is not an odd multiple of $a$ and lies below the top family
of Lemma~\ref{lem:windows}, while for $k$ even $\delta-a=(k-2)a$ fits neither shape of
Lemma~\ref{lem:windows}(v). If instead $b\le2a$, Lemma~\ref{lem:idiff} gives
$\mathcal I=[a,b]$ and hence $\bar d\in[0,\delta]$ with $\delta\le a$, whereas
\eqref{eq:R1} forces $\bar d>a$. Either way $d\le b$.

\emph{Claim 4 \textup{(}the case $a<d\le b$\textup{)}.} Here $Q=0$ and $\bar d=d$, so \eqref{eq:R1} reads
$d-a\in\Omega$. Put $\beta:=b\bmod d$ and $m:=\lfloor b/d\rfloor\ge1$. If $\beta<a$ then
$0,\beta\in[0,a)\subseteq Y$ and $\beta\equiv0+b\pmod d$, contradicting
$(Y+b)\cap Y=\emptyset$; so $\beta\ge a$. Next, $d-a$ cannot lie in the top family of
Lemma~\ref{lem:windows}: for $\varepsilon\ge1$ that family forces $d\ge a+b-\varepsilon>b$ or $d=a+b$, and for $k$ even it forces $d\in[ka,(k+1)a]$, whence $d=b$ and $\beta=0$, already
excluded. So $d-a$ is an odd multiple of $a$, say $d=2ja$ with $j\ge1$. Finally put
$\nu:=d-\beta=(m+1)d-b\le d-a$. For $x\in[0,a)\subseteq Y$ the integer $x+\nu$ lies in
$[0,d)$ and differs from $x$ by $-\beta\equiv-b$ in $\Z/d\Z$, so $[\nu,\nu+a)\cap
Y=\emptyset$; as $\nu+a\le d<a+b$ these are literal residues and $\nu\in\Omega$. Substituting
$d=2ja$ gives $\nu=2j(m+1)a-b$, and we test the two shapes of Lemma~\ref{lem:windows}. If
$\nu$ is an odd multiple $(2i-1)a$ of $a$, then $b=(2j(m+1)-2i+1)a$ exhibits $b$ as an odd
multiple of $a$, excluded in the present shapes. If $\nu$ lies in the top family, then: for
$\varepsilon\ge1$ and $\eta$ even, $\nu=(\eta+1)a+u$ with $0\le u\le\varepsilon$ gives
$a\mid(\varepsilon+u)$, and $1\le\varepsilon+u\le2a-2$ forces $\varepsilon+u=a$ and then
$2j(m+1)=2\eta+3$, odd; for $\varepsilon\ge1$ and $\eta$ odd, $\nu=(\eta+1)a+\varepsilon$
gives $a\mid2\varepsilon$, forcing $2\varepsilon=a$ and again $2j(m+1)=2\eta+3$; for $k$
even, $\nu=(k-1)a+u$ with $0\le u\le a$ gives $u\in\{0,a\}$, and $u=0$ makes $2j(m+1)=2k-1$
odd while $u=a$ makes $\nu=ka=b>d-a\ge\nu$. \textup{(}For $\eta=0$ no $d$ survives
\eqref{eq:R1} at all: $\Omega=[a,b]$ forces $d\ge2a>b\ge d$.\textup{)} Claims~3
and~4 show that $\bar d\neq0$ is impossible outside the harmonic shape, where $\mu=a+b$ and $\bar d=0$ gives $(a+b)\mid d$; in the harmonic shape Claim~1 gave
$2a\mid d$. Hence $\mu\mid d$ in every case.
\end{proof}


\section{Deferred proofs for the additive family}\label{app:additive}

We establish the explicit nim-value function used in the additive
case of Theorem~\ref{thm:lfunction}, starting from the $\P$-position
formula of Theorem~\ref{thm:additive}.
We first show that the four translates of $\W_0$ cover every residue
modulo $p$, then verify that the function obtained by assigning
the smallest translate index satisfies the mex recurrence.

Write $p=p_{\mathrm{add}}$ and
$W=\W_0(S)\cap[0,p)=\{\alpha_n:0\le n<Q\}$.
The bound on the largest $\P$-position in Theorem~\ref{thm:additive} gives
\begin{equation}\label{eq:wmax}
  \max W=\alpha_{Q-1}=p-c-1.
\end{equation}
In particular $W$, $W+a$, $W+b$ and $W+c$ are subsets of $[0,p)$,
without reduction modulo $p$. For odd $\eta$, one has
$N=m_\eta=\tfrac{\eta+1}{2}a+\varepsilon$ and $Q=am_\eta$;
for even $\eta$, $Q=N$ and $W=\mathcal Z_{\mathrm a}$.
The same proof gives the literal difference count
\[
  \#\{w\in W:w+\delta\in W\}
  =\begin{cases}
    a-\varepsilon,&\eta\text{ even},\\
    a\varepsilon,&\eta\text{ odd}.
  \end{cases}
\]
These are the two counts used below.

\begin{proof}[Additive case of Theorem~\ref{thm:lfunction}]
We now identify the nim-value function for the additive family. Note $\gcd(a,b)=1$: it divides $c=a+b$, and $\gcd(a,b,c)=1$. Let
$p$ be the
period and $\delta=b-a$. Consider, modulo $p$, the four translates of $\W_0$ by the shifts
$(o_0,o_1,o_2,o_3)=(0,\,a,\,-b,\,-\delta)$ and put
\[
  F(x)=\min\{\,i\in\{0,1,2,3\}:\ x\in\W_0+o_i\ (\mathrm{mod}\ p)\,\}.
\]
We prove $F=\G_S$; for the primitive quadratic case $\eta=1$ this is the nim-value
determination of \cite{Larsson}, and the argument below is uniform in $\eta$.

\emph{Disjointness and covering.} Translates $\W_0+o_i,\W_0+o_j$ meet iff $o_j-o_i$ is,
modulo $p$, a difference of two elements of $\W_0$. No two elements of $\W_0$
differ by $a$, $b$ or $c=a+b$ modulo $p$: literal move-differences are impossible
for a $\P$-set, and differences reduced modulo $p$ are excluded because
$\W_0\cap[0,p)\subseteq[0,\,p-c)$, established in the proof of
Theorem~\ref{thm:additive} for $\eta$ odd and immediate from
$\mathcal Z_{\mathrm a}\subseteq[0,(\eta+1)a)$ with $p-c=(\eta+1)a$ for $\eta$
even. Among the gaps of these translates
\[
  o_1-o_0=a,\quad o_2-o_0=-b,\quad o_2-o_1=-c,\quad o_3-o_1=-b,\quad o_3-o_2=a
\]
each is forbidden, while $o_3-o_0=-\delta$ is unconstrained. Hence the four translates are
pairwise disjoint with the single possible exception $\W_0\cap(\W_0-\delta)$, and a triple
coincidence is impossible \textup{(}it would force a second forbidden gap\textup{)}. Writing
$N_\delta:=|\W_0\cap(\W_0-\delta)\cap[0,p)|$, inclusion--exclusion gives, per period,
$\bigl|\bigcup_i(\W_0+o_i)\bigr|=4\,|\W_0\cap[0,p)|-N_\delta$, so the four translates cover
$[0,p)$ iff
\begin{equation}\label{eq:cover}
  N_\delta=4\,|\W_0\cap[0,p)|-p .
\end{equation}
Both sides of \eqref{eq:cover} are computed from Theorem~\ref{thm:additive}. For $\eta$ odd,
$|\W_0\cap[0,p)|=am$ with $m=m_\eta$ and $p=a(2m+b)$, so
$4am-p=a(2m-b)=a\varepsilon$; and the proof of Theorem~\ref{thm:additive} established
$N_\delta=a\varepsilon$ \textup{(}the count of $w\in W$ with $w+\delta\in W$; no
reduction modulo $p$ occurs, since $w+\delta\le(p-c-1)+\delta<p$ by \eqref{eq:wmax}\textup{)}. For
$\eta$ even, $|\W_0\cap[0,p)|=(\tfrac{\eta}{2}+1)a$ and $p=(2\eta+3)a+\varepsilon$, so
$4(\tfrac{\eta}{2}+1)a-p=a-\varepsilon$; and $N_\delta=a-\varepsilon$ directly: for
$x=2ja+r\in\mathcal Z_{\mathrm a}$ \textup{(}$0\le j\le\eta/2$, $0\le r<a$\textup{)} one has
$x+\delta=(2j+\eta)a+(r+\varepsilon)$ with no reduction modulo $p$
\textup{(}$x+\delta\le(2\eta+1)a+\varepsilon-1<p$\textup{)}; this lies in an odd block when
$r+\varepsilon\ge a$, and otherwise in the even block of index $2j+\eta$, which belongs to
$\mathcal Z_{\mathrm a}$ only when $2j+\eta\le\eta$, i.e.\ $j=0$. This leaves exactly the
$a-\varepsilon$ values $r<a-\varepsilon$ at $j=0$. In both parities \eqref{eq:cover} holds,
so the four translates cover $[0,p)$ and $F$ is everywhere defined. Their priority regions are
\[
  R_0=\W_0,\quad R_1=\W_0+a,\quad R_2=\W_0-b,\quad R_3=(\W_0-\delta)\setminus\W_0 .
\]

\emph{The recurrence.} We show $F=\G_S$ by strong induction on $x$, using at each step the
induction hypothesis $\G_S=F$ below $x$, Lemma~\ref{lem:ferguson} (value $1\iff x\in\W_0+a$),
and the collision property just recorded. The known $\P$-position set and Ferguson's pairing already
determine the positions of values $0$ and $1$.
On the remaining regions, it is enough to decide whether a
value-$2$ option exists: there is none on $R_2$, whereas the
$a$-option has value $2$ on $R_3$.
\begin{itemize}
\item $x\in R_0=\W_0$: then $\G_S(x)=0=F(x)$ by the definition of $\W_0$.
\item $x\in R_1=\W_0+a$ (so $x\notin\W_0$): Lemma~\ref{lem:ferguson} gives $\G_S(x)=1=F(x)$.
\item $x\in R_2=\W_0-b$: from $x\notin\W_0$ (else $x,x+b\in\W_0$ differ by $b$) and
$x\notin\W_0+a$ (else $x-a,x+b\in\W_0$ differ by $c$) we get $\G_S(x)\in\{2,3\}$ via
Lemma~\ref{lem:ferguson}. No option lies in $\W_0-b$, since $x-s\in\W_0-b$ for $s\in\{a,b,c\}$
would place $x-s+b$ and $x+b$ in $\W_0$ a distance $s$ apart; by the induction hypothesis the
value-$2$ positions below $x$ are exactly $\W_0-b$, so no option has value $2$. As $\G_S(x)=3$
would require an option of value $2$, we conclude $\G_S(x)=2=F(x)$.
\item $x\in R_3=(\W_0-\delta)\setminus\W_0$: from $x\notin\W_0+a$ (else $x-a,x+\delta\in\W_0$
differ by $b$) and $x\notin\W_0-b$ (else $x+b,x+\delta\in\W_0$ differ by $a$) we again get
$\G_S(x)\in\{2,3\}$. The $a$-option satisfies $x-a\in\W_0-b$: indeed $x+\delta\in\W_0$ gives
$(x-a)+b=x+\delta\in\W_0$. By the induction hypothesis $\G_S(x-a)=2$, so value $2$ is an
option and $\G_S(x)\neq2$; hence $\G_S(x)=3=F(x)$.
\end{itemize}
This proves $F=\G_S$. In particular the value-$2$ set is $\W_2=\W_0-b$ and the value-$3$ set
is $\W_3=(\W_0-\delta)\setminus\W_0$; their identification with the $L$-parity
candidates $V_2,V_3$ of Definition~\ref{def:L} is carried out in the proof of
Theorem~\ref{thm:lfunction}.
\end{proof}

\section*{Acknowledgements}

This work grew out of a research visit to the Indian Institute of Technology
Bombay. The questions treated here originate in the collaboration with Urban
Larsson on additive subtraction games \cite{Larsson} begun during that visit,
and the present paper is the author's own continuation and development of that
starting point. The author is grateful to Urban Larsson for his hospitality and
for valuable comments on an earlier version of this paper.

\section*{Generative AI disclosure}
During the preparation of this work the author used Anthropic's Claude (Claude Code;
models Claude Fable~5 and Fable~5.1) for numerical exploration and verification scripts, for an
accompanying Lean~4 formalization \textup{(}its coverage is documented in the
ancillary files\textup{)}, and for drafting and revising
the exposition and the figures, and OpenAI's Codex (models GPT-5.6~Sol and GPT-6 Astra) for adversarial manuscript
review and revision of the exposition.
The principal mathematical ideas and the analysis are the author's. The author reviewed and verified every statement and all AI-assisted
output, and takes full responsibility for the manuscript and its claims. The
verification code and the Lean formalization accompany this preprint as ancillary
files; the Lean declarations and proofs are those of commit \texttt{3c3d49bf}.



\begin{thebibliography}{99}
\bibitem{AB95} I.~Alth\"ofer, J.~B\"ultermann, Superlinear period lengths in some subtraction
  games, \emph{Theoret. Comput. Sci.} \textbf{148} (1995), 111--119.
\bibitem{berlekamp2004winning} E.~R. Berlekamp, J.~H. Conway, R.~K. Guy,
  \emph{Winning Ways for Your Mathematical Plays}, 2nd ed., A K Peters, 2001--2004.
\bibitem{BLMY} A.~Bhagat, U.~Larsson, H.~Manabe, T.~Yamashita, \emph{Additive sink subtraction},
  arXiv:2601.18715, 2026.
\bibitem{Ferguson} T.~S. Ferguson, On sums of graph games with last player losing,
  \emph{Internat. J. Game Theory} \textbf{3} (1974), 159--167.
\bibitem{FineWilf} N.~J. Fine, H.~S. Wilf, Uniqueness theorems for periodic functions,
  \emph{Proc. Amer. Math. Soc.} \textbf{16} (1965), 109--114.
\bibitem{Flammenkamp} A.~Flammenkamp, \emph{Lange Perioden in Subtraktions-Spielen},
  Ph.D. thesis, Universit\"at Bielefeld, 1997.
\bibitem{golomb1966mathematical} S.~W. Golomb,
  A mathematical investigation of games of ``take-away'', \emph{J. Combin. Theory} \textbf{1} (1966), 443--458.
\bibitem{Grundy} P.~M. Grundy, Mathematics and games, \emph{Eureka} \textbf{2} (1939), 6--8.
\bibitem{Ho} N.~B. Ho, On the expansion of three-element subtraction sets,
  \emph{Theoret. Comput. Sci.} \textbf{582} (2015), 35--47.
\bibitem{Larsson} U.~Larsson, H.~Manabe, \emph{Additive Subtraction Games}, arXiv:2603.10414, 2026.
\bibitem{LarssonSaha} U.~Larsson, I.~Saha, A brief conversation about
subtraction games, with an appendix by K.~Suetsugu, in \emph{Games of No Chance 6},
Math. Sci. Res. Inst. Publ. \textbf{71}, Cambridge University Press (2025),
25--42.
\bibitem{MiklosPost} I.~Mikl\'os, L.~Post, Superpolynomial period lengths of the winning
  positions in the subtraction game, \emph{Internat. J. Game Theory} \textbf{53} (2024),
  1275--1313; extended version arXiv:2312.02426v1. Result numbers cited in this paper
  follow the arXiv version.
\bibitem{MoriwakiThesis} Y.~Moriwaki, \emph{Results on Mory sequences of Restricted
  Nim, including the proof of Conway's Folklore Theorem}, Master's thesis, Hiroshima
  University, 2025 (in Japanese).
\bibitem{Moriwaki} Y.~Moriwaki, \emph{Subtraction Nim with Continuous Parameters},
  arXiv:2606.11658, 2026.
\bibitem{Sprague} R.~P. Sprague, \"Uber mathematische Kampfspiele,
  \emph{T\^ohoku Math. J.} \textbf{41} (1935--36), 438--444.
\bibitem{Ward} M.~Ward, \emph{A Conjecture about Periods in Subtraction Games}, arXiv:1606.04029, 2016.
\bibitem{Zhang} S.~Zhang, On the linearity of the periods of subtraction
games, \emph{Theoret. Comput. Sci.} \textbf{985} (2024), 114350.
References to results and corrections follow the
\href{https://zhangshenxing.github.io/publications/Zhang2024tcs\%20On\%20the\%20linearity\%20of\%20the\%20periods\%20of\%20subtraction\%20games.pdf}{author's revised version of 30 July 2025}.

\end{thebibliography}
\end{document}